\documentclass[11pt,twoside,a4paper,notitlepage]{article}

\usepackage{amsmath}

\usepackage{amssymb}
\usepackage{amsfonts}
\usepackage[a4paper,margin=2cm,vmargin={1cm,2cm},includeheadfoot]{geometry}
\usepackage[T1]{fontenc}
\usepackage{amsthm}
\usepackage{enumerate}
\usepackage[scaled=0.92]{helvet}
\usepackage{srcltx}
\usepackage{graphicx}
\usepackage{tikz}
\usepackage{booktabs}
\usepackage{subfigure}

\usepackage[labelfont=bf,margin=1cm,justification=justified,labelsep=period]{caption}

\usepackage{amsmath,amssymb,amsthm,soul,setspace}
\usepackage[nottoc,numbib]{tocbibind}

\theoremstyle{plain}
\newtheorem{theorem}{Theorem}[section]

\newtheorem{lemma}[theorem]{Lemma}
\theoremstyle{definition}

\newtheorem{remark}[theorem]{Remark}

\newcommand{\E}[1]{\mathrm{E} \Big[ #1 \Big]}
\def\R{\mathbb{R}}
\def\N{\mathbb{N}}

\def\F{\mathcal{F}}

\usepackage{authblk}

\title{SPDEs with fractional noise in space: continuity in law \\ with respect to the Hurst index}

\author[1]{Luca M. Giordano} 
\author[2]{Maria Jolis}
\author[3]{Llu\'is Quer-Sardanyons \thanks{Corresponding author}}

\affil[1]{Department of Mathematics, University of Milano,
Via C. Saldini 50, 20133 Milano, Italy, and
Departament de Matem\`atiques, Universitat
Aut\`{o}noma de Barcelona, 08193 Bellaterra, Catalonia,
Spain. E-mail address: luca.giordano@unimi.it} 

\affil[2,3]{Departament de Matem\`atiques, Universitat
Aut\`{o}noma de Barcelona, 08193 Bellaterra, Catalonia,
Spain. E-mail addresses: mjolis@mat.uab.cat, quer@mat.uab.cat.}

\date{\today}

\begin{document}

\maketitle

\abstract{
In this article, we consider the quasi-linear stochastic wave and heat equations on the real line and
with an additive Gaussian noise which is white in time and behaves in space like a fractional Brownian
motion with Hurst index $H\in (0,1)$. The drift term is assumed to be globally Lipschitz. We prove that the
solution of each of the above equations is continuous in terms of the index $H$, with respect to the
convergence in law in the space of continuous functions.
}

\bigskip

\noindent {MSC 2010:} 60B10; 60G60; 60H15


\medskip

\noindent {Keywords:} stochastic wave equation; stochastic heat equation; weak convergence; fractional noise


\section{Introduction}
\label{ch: introduction}

We consider the following stochastic wave and heat equations on
$[0,\infty)\times \mathbb{R}$, respectively:
\begin{equation}
\label{eq: wave} \tag{SWE}
\begin{cases}
\dfrac{\partial^2 u^H}{\partial t^2} (t,x)=\dfrac{\partial^2 u^H}{\partial x^2}(t,x)+b(u^H(t,x))+
\dot{W}^H(t,x),\\
u(0,x)=u_0(x), \quad x\in\R, \\
u_t(0,x)=v_0(x), \quad x\in\R, \\
\end{cases}
\end{equation}
and
\begin{equation}
\label{eq: heat}
\tag{SHE}
\begin{cases}
\dfrac{\partial u^H}{\partial t} (t,x)=\dfrac{\partial^2 u^H}{\partial x^2} (t,x)+b(u^H(t,x))+\dot{W}^H(t,x), \\
u(0,x)=u_0(x), \quad x\in\R. \\
\end{cases}
\end{equation}
The initial conditions $u_0,v_0:\R\rightarrow \R$ are deterministic
measurable functions which satisfy some regularity conditions
specified below. The drift coefficient $b:\R\rightarrow \R$ is
assumed to be globally Lipschitz.

The term $\dot{W}^H(t,x)$ stands for a random perturbation that is
supposed to be a Gaussian noise which is white in time and has a
spatially homogeneous correlation of fractional type. More
precisely, on some complete probability space
$(\Omega,\F,\mathbb{P})$, the noise $\dot{W}^H$ is defined by a
family of centered Gaussian random variables $\{W^H(\varphi), \,
\varphi \in \mathcal{D}\}$, where
$\mathcal{D}:=C^\infty_0([0,\infty)\times \mathbb{R})$ is the space
of infinitely differentiable functions with compact support, with
covariance functional
\begin{equation}
\E{W^H(\varphi)W^H(\psi)}=\int_{0}^{\infty} \int_{\R}
\F\varphi(t,\cdot)(\xi)\overline{\F\psi(t,\cdot)(\xi)}\,
\mu_H(d\xi)dt,
\end{equation}
for all $\varphi,\psi \in \mathcal{D}$, where $\F$ denotes the
Fourier transform in the space variable. For any $H\in (0,1)$, the
{\it{spectral}} measure $\mu_H$ is given by
\begin{equation}
\label{eq: spectral measure and c_H} \mu_H(d\xi):=c_H
|\xi|^{1-2H}d\xi, \qquad c_H=\frac{\Gamma(2H+1)\sin(\pi H)}{2\pi}.
\end{equation}

The above covariance relation, as in \cite{rev}, is used to construct an
inner product on the space $\mathcal{D}$ defined in the following
way:
$$ \left\langle \varphi, \psi \right\rangle_H:=  \E{W^H(\varphi)W^H(\psi)}, \quad \varphi, \psi\in
\mathcal{D}.$$
Let $\mathcal{H}^H$ be the completion of $\mathcal{D}$
with respect to the inner product
$\langle\cdot,\cdot\rangle_H$, which will be the natural
space of deterministic integrands with respect to $W^H$. Indeed, our
noise can be extended to a centered Gaussian family $\{W^H(g),\,
g\in \mathcal{H}^H\}$ indexed on the Hilbert space $\mathcal{H}^H$ and
satisfying
\[
\E{W^H(g_1)W^H(g_2)}=\langle g_1,g_2\rangle_H.
\]
As usual, for any $g\in \mathcal{H}^H$, we say that $W^H(g)$ is the
Wiener integral of $g$ and we denote it by
$$\int_{0}^{\infty}\int_{\R}g(t,x)W^H(dt,dx):=W^H(g).$$
The space $\mathcal{H}^H$ contains all functions $g$ such that its
Fourier transform in the space variable satisfies (see \cite[Thm.
2.7]{1/4<H<1/2} and \cite[Prop. 2.9]{rev}):
$$\int_0^{\infty} \int_{\R} |\F g(t,\cdot)(\xi)|^2|\xi|^{1-2H}d\xi dt<\infty.$$
In particular, the space $\mathcal{H}^H$ contains all elements of the
form $1_{[0,t]\times[0,x]}$, with $t>0$ and $x\in\R$. Then, the
following random field is naturally associated to our noise $W^H$:
$$X^H(t,x):=W^H\big(1_{[0,t]\times[0,x]}\big).$$
As a consequence of the representation in law of the fractional
Brownian motion as a Wiener type integral with respect to a complex
Brownian motion (see, for instance, \cite[p. 257]{PT}), we have that

\begin{equation*}
\begin{split}  \E{X^H(t,x)X^H(s,y)}
& = \int_{0}^{\infty}  \int_{\R} \F 1_{[0,t]\times[0,x]}(r,\cdot)(\xi)
\overline{\F 1_{[0,s]\times[0,y]}(r,\cdot)(\xi)}\, \mu_H(d\xi)dt \\
& =  \int_{0}^{t\wedge s} \int_{\R} \F 1_{[0,x]}(\xi)\overline{\F 1_{[0,y]}(\xi)}\, \mu_H(d\xi)dt \\
& = \frac{1}{2}(t\wedge s)\Big(|x|^{2H}+|y|^{2H}-|x-y|^{2H}\Big). \\
\end{split}
\end{equation*}
This is the covariance of a standard Brownian motion in the time
variable, while in the space variable we have obtained the
covariance of a fractional Brownian motion with Hurst parameter $H$.

We denote by $(\F^H_t)_{t\ge 0}$ the filtration generated by $W^H$,
namely
\begin{equation}
    \F^H_t:=\sigma( W^H(1_{[0,s]}\varphi),\, s\in[0,t],\,\varphi \in C^\infty_0(\R))\, \lor \, \mathcal{N},
    \label{eq:24}
    \end{equation}
    where $\mathcal N$ denotes the class of $\mathbb{P}$-null sets in $\mathcal F$.

The solution to equations \eqref{eq: wave} and \eqref{eq: heat} will
be interpreted in the {\it{mild}} sense. That is, for any $T>0$, we
say that an adapted and jointly measurable process
$u^H=\{u^H(t,x),\, (t,x)\in [0,T]\times \R\}$ solves \eqref{eq:
wave} (resp. \eqref{eq: heat}) if, for all $(t,x)\in[0,T]\times \R$,
it holds
    \begin{equation}
    \begin{split}
    u^H(t,x) & = I_0(t,x)+ \int_{0}^{t}\int_{\R} G_{t-s}(x-y) W^H(ds,dy)  \\
     & \qquad +  \int_{0}^{t}\int_\R G_{t-s}(x-y) b(u^H(s,y)) dy ds, \quad \mathbb{P}\text{-a.s.}
    \end{split}
    \label{def: mild solution}
    \end{equation}
Here, the function $G_{t}(x)$ is the fundamental solution of the
wave (resp. heat) equation in $\R$, and $I_0(t,x)$ is the solution
of the corresponding deterministic linear equation. These are given
by
\begin{equation}
I_0(t,x)=\begin{cases}
\frac{1}{2}\int_{x-t}^{x+t}v_0(y)dy+\frac{1}{2}\Big(u_0(x+t)-u_0(x-t)\Big), & \text{wave equation,} \\
\\
\int_{\R} G_t(x-y)u_0(y)dy, & \text{heat equation,} \\
\end{cases}
\label{eq:0}
\end{equation}
and
\begin{equation}
G_t(x)=\begin{cases}
\frac{1}{2}1_{|x|<t}(x), & \text{wave equation,} \\
\\
\dfrac{1}{(2\pi t)^{1/2}}\exp\Big(-\dfrac{|x|^2}{2t}
\Big), & \text{heat equation.} \\
\end{cases}
\label{eq:1}
\end{equation}

\medskip

Our main objective consists in studying the continuity in law, in
the space $C([0,T]\times \R)$ of continuous functions, of the
solution $u^H$ to equations \eqref{eq: wave} and \eqref{eq: heat}
with respect to the Hurst index $H\in (0,1)$. More precisely, we fix
$H_0\in (0,1)$ and we will provide sufficient conditions on the
initial data under which, whenever $H\rightarrow H_0$, the
$C([0,T]\times \R)$-valued random variable $u^H$ converges in law to
$u^{H_0}$ (cf. Theorem \ref{thm:9}). Recall that the parameter $H$
quantifies the regularity of the random perturbation $W^H$, and
hence the level of noise in the system. So we will study the
probabilistic behavior of the solution in terms of $H$, aiming at
showing that the sensitivity in $H$ implies the corresponding
convergence of the solutions.

We note that continuity in law with respect to fractionality
indices has been studied in other related contexts. We refer the
reader to \cite{JV4,JV3,JV2} for results involving symmetric, Wiener
and multiple integrals with respect to fractional Brownian motion,
respectively, while in \cite{JV1,Xiao} the convergence in law of the
local time of the fractional Brownian motion and of anisotropic
Gaussian random fields has been considered, respectively. Finally,
in the recent paper \cite{Ait}, the continuity in law for some
additive functionals of the sub-fractional Brownian motion has been
studied.

In order to tackle our main objective, we start by focusing on the
linear version of equations \eqref{eq: wave} and \eqref{eq: heat}.
That is, we consider the case where $b=0$. Here, we first prove
existence and uniqueness of solution, together with the existence of
a continuous modification, for any $H\in (0,1)$ (cf. Theorem
\ref{thm:1}). So, for the particular case of \eqref{eq: wave} and
\eqref{eq: heat}, this result puts together the more general ones of
\cite{BalanJFAA} (valid for $H\leq \frac12$) and \cite{extend}
(valid for $H>\frac12$). The convergence in law of $u^H$ to
$u^{H_0}$ reduces to analyze the convergence of the corresponding
stochastic convolutions, which are centered Gaussian processes. For
this, we first check that the corresponding family of probability laws
is tight in the space $C([0,T]\times \R)$, and then we
identify the limit law by characterizing the underlying Gaussian
candidate for the limit (see Theorem \ref{th: convergence linear
additive case} for details). Finally, we point out that in the
linear case, the proof of the main convergence result holds for both
wave and heat equations.

We remark that there are several well-posedness results for
equations \eqref{eq: wave} and \eqref{eq: heat} with $b=0$ and a
more general noise term, namely of the form $\sigma(u(t,x))
\dot{W}^H(t,x)$, for some function $\sigma:\R\rightarrow \R$: if
$H<\frac12$, we refer the reader to, e.g., \cite{1/4<H<1/2,Tindel},
while the case $H\geq \frac12$ falls in the general framework of
Walsh and Dalang (see \cite{walsh,extend,rev}). When $H<\frac12$,
most of the existing work focuses on the particular coefficient
$\sigma(z)=z$, which corresponds to the so-called Hyperbolic
Anderson Model (HAM) and the Parabolic Anderson Model (PAM),
respectively (see \cite{1/4<H<1/2,Hu,BJQ} and references therein).
In these cases, the fact that $H<\frac12$ entails important
technical difficulties in order to define stochastic integrals with
respect to the noise $W^H$. Moreover, as proved in
\cite[Prop.3.7]{BJQ}, the above equations admit a unique solution if
and only if $H>\frac14$. In the present article, we do not encounter
such issues since the noise appears in the equations in an additive
way. Indeed, we plan to address the convergence in law with respect
to $H$ for the HAM and PAM in a separate publication, where the
underlying stochastic integrals are interpreted in the Skorohod
sense.

We turn now to the study of the quasi-linear case, that is assuming
that $b$ is a general Lipschitz function. Here, we first prove that
equations \eqref{eq: wave} and \eqref{eq: heat} admit a unique
solution (see Theorem \ref{th: existence and uniqueness b
Lipschitz}). This result holds for any $H\in (0,1)$ and, as far as
we know, is new for the case $H<\frac12$ (if $H>\frac12$, it follows
from \cite[Thm. 4.3]{rev}). Moreover, we note that the proof of
Theorem \ref{th: existence and uniqueness b Lipschitz} can be built
in a unified way for both wave and heat equations.

Nevertheless, the analysis of the weak convergence in the
quasi-linear case does not admit a unified proof for wave and heat
equations. More precisely, for the wave equation, the convergence in
law of $u^H$ to $u^{H_0}$, whenever $H\rightarrow H_0$, follows from
a pathwise argument: we prove that, for almost all $\omega$, the
solution of \eqref{eq: wave} can be seen as the image of the
stochastic convolution through a certain continuous functional $F:
C([0,T]\times \R)\rightarrow C([0,T]\times \R)$. In the case of the
heat equation, this argument cannot be directly applied, for the
associated deterministic equation which has to be solved in order to
define the above-mentioned functional is not well-posed for a
general coefficient $b$. We overcome this difficulty by first
assuming that $b$ is a bounded function and then using a truncation
argument. As it will be exhibited in Section \ref{subsec: b unbounded},
this part of the paper contains most of the technical difficulties
that we need to face. It is also worthy to point out that, in the
analysis of the wave equation and the heat equation with bounded
$b$, we have established ad hoc versions of Gr\"onwall lemma which
have been crucial to complete the corresponding proofs (see,
respectively, Lemmas \ref{lemma: Gronwall wave} and \ref{lemma:
gronwall heat}).

This article is organized as follows. Section \ref{ch: linear additive} is devoted to study the convergence in law
for equations \eqref{eq: wave} and \eqref{eq: heat} in the linear additive case (i.e. $b=0$).
In Section \ref{ch: quasi-linear additive}, existence, uniqueness and pathwise H\"older continuity in the
quasi-linear additive case are established. Finally, the main result on weak convergence for the
quasi-linear case is proved in Section \ref{sec:1}: here we treat separately the case of the wave equation
(Section \ref{sec:w}), the heat equation with $b$ bounded (Section \ref{subsec: heat b bounded}) and
the heat equation with general $b$ (Section \ref{subsec: b unbounded}).

When we make use of the constant $C$, we are meaning that the value
of that constant is not relevant for our computations, and also that
it can change its value from line to line. When two constants
(possibly different) appear on the same line, we will call them
$C_1,C_2$. Sometimes we use $C_p$ when we want to stress that the
constant depends on some exponent $p$.


\section{Weak convergence for the linear additive case}
\label{ch: linear additive}

In this section, we consider equations \eqref{eq: wave} and
\eqref{eq: heat} in the case where the drift term vanishes, that is
$b=0$. Then, the mild formulation \eqref{def: mild solution} reads
\begin{equation}
\label{eq: mild formulation linear additive}
    u^H(t,x)=  I_0(t,x)+ \int_{0}^{t}\int_{\R} G_{t-s}(x-y)W^H(ds,dy),
\end{equation}
where we recall that the term $I_0$ and the fundamental solution $G$ have been defined in
\eqref{eq:0} and \eqref{eq:1}, respectively. Throughout this section we assume that $H\in (0,1)$.
Notice that \eqref{eq: mild formulation linear additive} is now an explicit formula for the solution $u^H$.
We consider the following hypotheses on the initial data:

\medskip

\noindent {\bf{Hypothesis A:}} It holds that
\begin{itemize}
 \item[(a)] {\it{Wave equation}}: $u_0$ is continuous and $v_0\in L^1_{loc}(\mathbb R)$.
 \item[(b)] {\it{Heat equation}}:  $u_0$ is continuous and bounded.
\end{itemize}

It can be easily verified that the above conditions on the initial
data imply that $I_0:\R_+\times \R\rightarrow \R$ is a continuous
function. On the other hand, the stochastic convolution in
\eqref{eq: mild formulation linear additive} is a well-defined
centered Gaussian random variable since, for any $(t,x)\in
[0,T]\times \R$,
\begin{align*}
 \mathrm{E} \left[\left|\int_{0}^{t}\int_{\R} G_{t-s}(x-y)W^H(ds,dy)\right|^2\right] &
 = \int_{0}^{t}\int_{\R} |\F G_{t-s}(\xi)|^2|\xi|^{1-2H}\,d\xi \,ds \\
 & \leq  \int_{0}^{T}\int_{\R} |\F G_{s}(\xi)|^2|\xi|^{1-2H}\,d\xi \,ds <\infty,
\end{align*}
where we have applied Lemma \ref{lemma: 3.1} below. Hence, we have the following result:

\begin{theorem}\label{thm:1}
 Assume that Hypothesis A holds and let $H\in (0,1)$. Then, there exists a unique solution
 $u^H=\{u^H(t,x),\, (t,x)\in [0,T]\times \R\}$ of equation \eqref{eq: mild formulation linear additive}.
 Moreover, the random field $u^H$ admits a modification with continuous sample paths.
\end{theorem}

\begin{proof}
 We only need to prove that $u^H$ has a modification with continuous paths. Indeed, since $I_0$ is deterministic and continuous,
 we check that the stochastic convolution
 $\tilde{u}^H(t,x):=u^H(t,x)-I_0(t,x)$
 admits
 a continuous modification. This is a direct consequence of Step 1 in the proof of Theorem
 \ref{th: convergence linear additive case} below. More precisely, for any $p\geq 2$, there
 exists a constant $C$ (depending only on $p$) such that, for all $t,t'\in [0,T]$ and
 $x,x'\in \R$, it holds
 \[
  \E{|\tilde{u}^H(t,x)-\tilde{u}^H(t',x')|^p} \leq C \left\{ |t-t'|^{\alpha p} + |x-x'|^{p H}\right\},
 \]
 where $\alpha= H$ for the wave equation and $\alpha= \frac H2$ for the heat equation.
 An application of Kolmogorov's continuity criterion concludes the proof.
\end{proof}

\begin{remark}\label{rmk:1}
In the case of the heat equation, the assumptions of Theorem
\ref{thm:1} indeed imply that, for all $p\geq 1$,
\[
\sup_{(t,x)\in [0,T]\times \R} \E{|u^H(t,x)|^p}<\infty.
\]
For the wave equation, this property can be obtained by slightly
strengthening the hypotheses of $u_0$ and $v_0$, e.g. assuming that
they are bounded functions (see \cite[Lem. 4.2]{rev}).
\end{remark}

\begin{remark}\label{rmk:2}
The proof of Theorem \ref{thm:1} implies that the stochastic
convolution in equation \eqref{eq: mild formulation linear additive}
has a modification which is (locally) $\beta_1$-H\"older continuous
in time for any $\beta_1\in (0,\alpha)$ and (locally)
$\beta_2$-H\"older continuous in space for any $\beta_2\in (0,H)$.
\end{remark}

In the proof of the main result of the present section (cf. Theorem
\ref{th: convergence linear additive case}), we will need the
following three technical lemmas (proved in \cite{1/4<H<1/2}). They
provide explicit estimates, depending on $H$, of the norm in the
space $L^2(\R;\mu^H)$ of the Fourier transforms of the fundamental
solutions of the deterministic wave and heat equations, where we
recall that, respectively:
\begin{equation}
\F G_t(\xi)=\dfrac{\sin(t|\xi|)}{|\xi|} \qquad \text{and} \qquad \F
G_t(\xi)=\exp\Big(\frac{-t\xi^2}{2}\Big), \quad t>0,\, \xi\in \R.
\label{eq:2}
\end{equation}
 In
the following three lemmas, we will denote either one of these two
functions by $\F G_t(\xi)$. We recall that the spatial spectral
measure is given by $\mu^H(d\xi)=c_H |\xi|^{1-2H} d\xi$ (see
\eqref{eq: spectral measure and c_H}).

\begin{lemma}[\cite{1/4<H<1/2}, Lemma 3.1.]
    \label{lemma: 3.1}
    Let $T>0$. Then,  the integral
    $$A_T(\alpha):=\int_{0}^{T}\int_{\R} |\F G_t(\xi)|^2|\xi|^\alpha\,d\xi \,dt$$
    converges if and only if $\alpha\in(-1,1)$. In this case, it holds:
    $$A_T(\alpha)=\begin{cases}
    2^{1-\alpha}C_\alpha\dfrac{1}{2-\alpha}T^{2-\alpha}  & \text{for the heat equation,}\\
    \\
    \dfrac{2}{1-\alpha}\Gamma\Big( \dfrac{\alpha+1}{2} \Big)T^{(1-\alpha)/2} & \text{for the wave equation,}\\
    \end{cases}$$
    where the constant $C_\alpha$ is given by
    $$C_{\alpha}=\begin{cases}
    \dfrac{1}{1-\alpha}\Gamma(\alpha)\sin(\pi\alpha/2), & \alpha\in(0,1),\\
    \\
    \dfrac{1}{\alpha}\dfrac{1}{1-\alpha}\Gamma(1+\alpha)\sin(\pi\alpha/2), & \alpha \in (-1,0),\\
    \\
    \dfrac{\pi}{2}, & \alpha=0.\\
    \end{cases}$$
\end{lemma}

\begin{lemma}[\cite{1/4<H<1/2}, Lemma 3.4.]
    \label{lemma: 3.4}
    Let $T>0$ and $\alpha\in(-1,1)$. Then, for any $h>0$, it holds:
    \begin{equation*}
    \int_{0}^{T}\int_{\R}(1-\cos(\xi h))\, |\F G_t(\xi)|^2|\xi|^\alpha\,d\xi \,dt \leq
    \begin{cases}
    C|h|^{1-\alpha} & \text{for the heat equation,} \\
    CT|h|^{1-\alpha} & \text{for the wave equation,} \\
    \end{cases}
    \end{equation*}
    where $C=\int_{\R} (1-\cos\eta)|\eta|^{\alpha-2}d\eta$.
\end{lemma}

\begin{lemma}[\cite{1/4<H<1/2}, Lemma 3.5.]
    \label{lemma: 3.5}
    Let $T>0$ and $\alpha\in(-1,1)$. Then, for any $h>0$, it holds:
    \begin{equation*}
    \int_{0}^{T}\int_{\R} |\F G_{t+h}(\xi)-\F G_{t}(\xi)|^2|\xi|^\alpha\,d\xi \,dt \leq
    \begin{cases}
    C_\alpha |h|^{(1-\alpha)/2} & \text{for the heat equation,} \\
    C_{\alpha} T|h|^{1-\alpha} & \text{for the wave equation,} \\
    \end{cases}
    \end{equation*}
    where
$$C_\alpha= \int_{\mathbb
R}\frac{(1-e^{-\eta^2/2})^2}{|\eta|^{2-\alpha}}d\eta \quad\text{for
the heat equation, and} $$
$$C_\alpha=4\int_{\mathbb R}\frac{\min(1,|\eta|^2)}{|\eta|^{2-\alpha}}d\eta \quad \text{for the wave equation.}
$$
\end{lemma}

\medskip
\medskip

We will also make use of the following tightness criterion in the
plane (see \cite[Prop. 2.3]{Yor}):
\begin{theorem}
    \label{th: centsov}
    Let $\{X_\lambda\}_{\lambda\in\Lambda}$ be a family of random functions indexed on the
    set $\Lambda$ and taking values in the space $C([0,T]\times \R)$, in which we consider
    the metric of uniform convergence over compact sets. Then, the family
    $\{X_\lambda\}_{\lambda\in \Lambda}$ is tight if, for any compact set $J\subset \R$,
    there exist $p',p>0$, $\delta>2$, and a constant $C$ such that the following holds
    for any $t',t\in [0,T]$ and $x',x\in J$:
    \begin{itemize}
        \item[(i)] $\sup_{\lambda\in \Lambda}\E{|X_\lambda(0,0)|^{p'}}<\infty$,
        \item[(ii)] $\sup_{\lambda\in \Lambda} \E{|X_\lambda(t',x')-X_\lambda(t,x)|^p}\leq C\Big(|t'-t|+|x'-x|\Big)^\delta$.
    \end{itemize}
\end{theorem}

\bigskip

We are now in position to state and prove the main result of this
section.

\begin{theorem}
    \label{th: convergence linear additive case}
    Consider a family $\{u^{H_n}\}_{n\geq 1}$ of solutions of equation $\eqref{eq: wave}$
    or \eqref{eq: heat}, and suppose that the Hurst indexes $H_n\to H_0\in (0,1)$, as $n\to\infty$.
    Then $u^{H_n}\xrightarrow{d}u^{H_0}$, as $n\to\infty$, where the convergence holds in distribution in the space
    $C([0,T]\times \R)$, where the latter is endowed with the metric of uniform convergence on
    compact sets.
\end{theorem}
\begin{proof}
We split the proof in two steps. In the first one, we prove that the
sequence of stochastic convolutions is tight in $C([0,T]\times
\R)$, while the second step is devoted to the identification of the
limit law.

\medskip

 \textbf{Step $\mathbf{1}$:} Since $H_n\to H_0$, the sequence $\{H_n\}$ is  contained in a compact set $K\subset(0,1)$.
    For a fixed $H\in(0,1)$, we have that the solution $u^H$ is expressed as
    \begin{equation*}
    u^H(t,x)=I_0(t,x)+\int_{0}^{t}\int_{\R} G_{t-s}(x-y)W^H(ds,dy).
    \end{equation*}
    We will apply Theorem \ref{th: centsov} to the family $\{\tilde{u}^H=u^H-I_0\}_{H\in K}$
    of stochastic convolutions:
    $$\tilde{u}^H(t,x)=u^H(t,x)-I_0(t,x)=\int_{0}^{t}\int_{\R}G_{t-s}(x-y)W^H(dy,ds).$$
     We write then, supposing without loss of generality that $t'\geq t$ and $x'\geq x$:
    \begin{equation*}
    \begin{split}
    \tilde{u}^H(t',x') - \tilde{u}^H(t,x) = &\int_{t}^{t'}\int_\R G_{t'-s}(x'-y)W^H(ds,dy)\\
    & \quad + \int_{0}^{t}\int_\R [G_{t'-s}(x'-y) - G_{t-s}(x-y)]W^H(ds,dy).
    \end{split}
    \end{equation*}
Thus, we have
    $$\E{|u(t,x)-u(t',x')|^p}\leq C_p(I_1+I_2),$$
    where $I_1,I_2$ are defined as:
    \begin{equation*}
    \label{eq: first term}
    I_1:=\mathrm{E}\left[\Big|\int_{t}^{t'}\int_\R
    G_{t'-s}(x'-y)W^H(ds,dy)\Big|^p\right],
    \end{equation*}
    \begin{equation*}
    \label{eq: second term}
    I_2:=\mathrm{E}\left[\Big|\int_{0}^{t}\int_\R [G_{t-s}(x-y)-G_{t'-s}(x'-y)]W^H(ds,dy)
    \Big|^p\right].
    \end{equation*}
 Since  $I_1$ is the moment of order $p$ of a centered Gaussian random variable,
 we have
  \begin{equation}
    \label{eq: estimate burkholder type}
    \begin{split}
    I_1& = \mathrm{E}\left[\Big| \int_{0}^{T}\int_{\R} 1_{[t,t']}(s)G_{t'-s}(x'-y) W^H(ds,dy)\Big|^p\right] \\
    & = z_p \, c_H^{p/2}\left[\int_{0}^{T}1_{[t,t']}(s)\int_{\R}|\mathcal{F}G_{t'-s}(x'-\cdot)(\xi)|^2|\xi|^{1-2H}d\xi\,ds\right]^{p/2} \\
    & =   z_p\, c_H^{p/2}\left[\int_{t}^{t'}\int_{\R}|\mathcal{F}G_{t'-s}(x'-\cdot)(\xi)|^2|\xi|^{1-2H}d\xi\,ds\right]^{p/2} \\
    & =  z_p\, c_H^{p/2}\left[\int_{0}^{t'-t}\int_{\R}|\mathcal{F}G_{s'}(\xi)|^2|\xi|^{1-2H}d\xi\,ds'\right]^{p/2}.
    \end{split}
    \end{equation}
    Notice that we have used the standard properties of Fourier transform in the space variable, and we performed the change of variable $s'=t'-s$.
    The constant $z_p$ is the $p$-order moment of a standard normal distribution and
    $c_H$ is given by (\ref{eq: spectral measure and c_H}).

     Now we apply Lemma \ref{lemma: 3.1} and obtain
    \begin{equation}
    \label{eq: estimate for I_1}
    I_1\leq \begin{cases}
    z_p\, c_H^{p/2} \Big[ 2^{2H}\tilde{C}_{1-2H}\dfrac{1}{1+2H}(t'-t)^{1+2H}\Big]^{p/2}, & \text{wave equation,} \\
    \\
    z_p c_H^{p/2} \Big[\frac{1}{H}\Gamma(1-H)(t'-t)^H\Big]^{p/2}, & \text{heat equation.}
    \end{cases}
    \end{equation}
    The above constant $\tilde{C}_{1-2H}$ is the one of Lemma \ref{lemma: 3.1}:
    $$\tilde{C}_{1-2H}=\begin{cases}
    \dfrac{1}{2H}\Gamma(1-2H)\sin\Big(\pi\dfrac{1-2H}{2}\Big),& H\in\Big(0,\dfrac{1}{2}\Big),\\
    \\
    \dfrac{1}{1-2H}\dfrac{1}{2H}\Gamma(2-2H)\sin\Big(\pi\dfrac{1-2H}{2}\Big), & H\in\Big(\dfrac{1}{2},1\Big),\\
    \\
    \dfrac{\pi}{2}, & H=\dfrac{1}{2}.\\
    \end{cases}$$
    First, we observe that
    $z_p$ is independent of $H$ and
    $$c_H=\dfrac{\Gamma(2H+1)\sin(\pi H)}{2\pi}\leq \dfrac{\Gamma(3)}{2\pi}=\dfrac{1}{\pi}.$$
    Next, as far as estimate \eqref{eq: estimate for I_1} for the wave
equation is concerned, we note that $2^{2H}\leq 4$ and
$\frac{1}{1+2H}\leq 1$, for any $H\in(0,1)$. Thus, we concentrate on
the constant $\tilde{C}_{1-2H}$, which we show that it is uniformly bounded in $H$. Clearly, the function
    $\tilde{C}_{1-2H}:(0,1)\to\R$
    has, possibly, a singularity only in $H=\frac{1}{2}$, but since $\Gamma(x)\sim \frac{1}{x}$
    as $x\to 0_+$, by simple calculations we have that the function $\tilde{C}_{1-2H}$
    is continuous also at the point $H=\frac{1}{2}$. Therefore,
    $\tilde{C}_{1-2H}$ is bounded on the set $K$.

    On the other hand, regarding estimate \eqref{eq: estimate for I_1} for the heat equation,
    we have that $\dfrac{1}{H}\Gamma(1-H)$ defines a continuous function of $H$
    on the interval $(0,1)$, and thus it is bounded on $K$.

    We now turn to the analysis of the term $I_2$. More precisely, we have
    \begin{equation*}
    \begin{split}
    I_2 & = \mathrm{E}\left[\Big|\int_{0}^{T}\int_\R 1_{[0,t]}(s)[G_{t-s}(x-y)-G_{t'-s}(x'-y)]W^H(ds,dy)\Big|^p\right] \\
     & = z_p\, c_H^{p/2} \left[\int_{0}^{T}1_{[0,t]}(s)\int_{\R}\Big|\mathcal{F}\Big(G_{t-s}(x-\cdot)-G_{t'-s}(x'-\cdot)\Big)(\xi)\Big|^2|\xi|^{1-2H}d\xi\,ds\right]^{p/2} \\
    & = z_p\, c_H^{p/2} \left[\int_{0}^{t}\int_{\R}\Big|\mathcal{F}G_{t-s}(x-\cdot)(\xi)-\mathcal{F}G_{t'-s}(x'-\cdot)(\xi)\Big|^2|\xi|^{1-2H}d\xi\,ds\right]^{p/2}  \\
    & \leq  z_p\, c_H^{p/2}C_p\Bigg(\left[\int_{0}^{t}\int_{\R}\Big|\mathcal{F}G_{t'-s}(x'-\cdot)(\xi)-\mathcal{F}G_{t-s}(x'-\cdot)(\xi)\Big|^2|\xi|^{1-2H}d\xi\,ds\right]^{p/2} \\
    & \qquad + \left[\int_{0}^{t}\int_{\R}\Big|\mathcal{F}G_{t-s}(x'-\cdot)(\xi)-\mathcal{F}G_{t-s}(x-\cdot)(\xi)\Big|^2|\xi|^{1-2H}d\xi\,ds\right]^{p/2}  \Bigg)\\
    & =  z_p c_H^{p/2}C_p\Big(J_1+J_2\Big),
    \end{split}
    \end{equation*}
where $C_p$ denotes some constant depending on $p$. We estimate
$J_1$ and $J_2$ using similar techniques as those used for the term
$I_1$. Hence, via the change of variable $s'=t-s$, we have:
    $$J_1=\left[\int_{0}^{t}\int_{\R}\Big|\mathcal{F}G_{s'+(t'-t)}(x'-\cdot)(\xi)-\mathcal{F}G_{s'}(x'-\cdot)(\xi)\Big|^2|\xi|^{1-2H}d\xi\,ds'\right]^{p/2}.$$
     Thus, by Lemma \ref{lemma: 3.5},
    \begin{equation*}
    \label{eq: estimate of J_1}
    J_1 \leq \begin{cases}
    M_H^{p/2}\, t^{p/2}(t'-t)^{pH}\leq M_H^{p/2}\, T^{p/2}(t'-t)^{pH}, & \text{wave equation,} \\
    \\
    N_H^{p/2}(t'-t)^{pH/2}, & \text{heat equation.} \\
    \end{cases}
    \end{equation*}
    The above constants are the following:
    \begin{equation*}
    \begin{split}
    \frac14M_H  & = \int_{\R} \dfrac{\min(1,|h|^2)}{|h|^{1+2H}}dh \\
    & =\int_{|h|>1} \dfrac{1}{|h|^{1+2H}}dh+\int_{|h|<1} \dfrac{1}{|h|^{2H-1}}dh \\
    & = \dfrac{1}{H}+\dfrac{1}{1-H},
    \end{split}
    \end{equation*}
    and
    \begin{equation*}
    \begin{split}
    N_H & = \int_{\R} \dfrac{(1-e^{-\frac{h^2}{2}})^2}{|h|^{1+2H}}dh\le\int_{\R} \dfrac{1-e^{-\frac{h^2}{2}}}{|h|^{1+2H}}dh
    \\
     &  \leq  \int_{|h|>1} \dfrac{1}{|h|^{1+2H}}dh+\int_{|h|<1} \dfrac{1}{|h|^{2H-1}}dh \\
    & = \dfrac{1}{H}+\dfrac{1}{1-H}.
    \end{split}
    \end{equation*}
    The function $H\mapsto \dfrac{1}{H}+\dfrac{1}{1-H}$
    is again continuous in $(0,1)$, and thus bounded for $H\in K$.

    For the term $J_2$, we have:
    \begin{equation*}
    \begin{split}
    J_2 & = \left[\int_{0}^{t}\int_{\R}\Big|\mathcal{F}G_{t-s}(x'-\cdot)(\xi)-\mathcal{F}G_{t-s}(x-\cdot)(\xi)\Big|^2|\xi|^{1-2H}d\xi\,ds\right]^{p/2} \\
     & = \left[\int_{0}^{t}\int_{\R}[1-\cos(\xi(x'-x))]\Big| \mathcal{F}G_{s'}(x-\cdot)(\xi)\Big|^2|\xi|^{1-2H}d\xi\,ds'\right]^{p/2}, \\
    \end{split}
    \end{equation*}
    and applying Lemma \ref{lemma: 3.4} we end up with
    $$J_2\leq
    \begin{cases}
    C_H^{p/2}\, t^{p/2}(x'-x)^{pH}\leq C_H^{p/2}\, T^{p/2}(x'-x)^{pH}  ,& \text{wave equation,} \\
    \\
    C_H^{p/2}(x'-x)^{pH} ,& \text{heat equation.}
    \end{cases}
    $$
    Here, the constant $C_H$ is
    \begin{equation*}
    C_H =  \int_{\R}\dfrac{1-cos(h)}{|h|^{1+2H}}dh \leq \dfrac{1}{H}+\dfrac{1}{1-H},
    \end{equation*}
    which again is a bounded function on the set $K$.

    To sum up, we have proved that
\[
  \E{|\tilde{u}^H(t,x)-\tilde{u}^H(t',x')|^p} \leq C \Big( (t'-t)^{\alpha p} + (x'-x)^{p H}\Big),
 \]
 where $\alpha= H$ for the wave equation and $\alpha= \frac H2$ for the heat equation, and the constant
 $C$ depends only of $p$ and $T$.
    Thus, choosing $p>\dfrac{4}{\min_{H\in K} H}$,
    we have that the hypotheses of Theorem \ref{th: centsov} are
    fulfilled  by the family $\{\tilde{u}^H\}_{H\in K}$, for both the solution
    to \eqref{eq: wave} and \eqref{eq: heat}. This concludes the
    first step of the proof.

\medskip

    \textbf{Step $\mathbf{2}$:} In order to identify the limit law
    of the sequence $\{u^{H_n}\}_{n\geq 1}$, we proceed to prove the
    convergence of the finite dimensional distributions of $\tilde{u}^{H_n}$ when $n\to \infty$.

    We recall that, for every $H\in(0,1)$, $\tilde{u}^H=u^H-I_0$ is
a centered Gaussian process, so it suffices to analyze the
convergence of the corresponding covariance functions.

Let $(t,x),(t',x')\in [0,T]\times\mathbb R$ and suppose that $t'\geq
t$. Then,
    \begin{equation*}
    \label{eq: covariance DCT start}
    \E{ \tilde{u}^{H_n}(t,x) \tilde{u}^{H_n}(t',x') }
    = c_{H_n} \int_{0}^{t} \int_{\R} \F G_{t-s}(x-\cdot)(\xi)
    \overline{\F G_{t'-s}(x'-\cdot)(\xi)} \, |\xi|^{1-2H_n}
    d\xi\,ds.
    \end{equation*}
    Let us first consider the case of the wave equation. Taking into
    account the explicit form of $\F G_t(\xi)$ (see \eqref{eq:2}), we have
    \begin{equation*}
    \begin{split}
    \E{ \tilde{u}^{H_n}(t,x) \tilde{u}^{H_n}(t',x') } = c_{H_n}  \int_{0}^{t} \int_{\R}
    \dfrac{e^{-i\xi(x-x')}\sin((t-s)|\xi|)\sin((t'-s)|\xi|)}{|\xi|^{1+2H_{n}}}\,d\xi\,ds.
    \end{split}
    \end{equation*}
    We clearly have that $c_{H_n}\rightarrow c_{H_0}$.
    The integrand function in the latter integral converges, as $n\rightarrow \infty$, to
    \[
    \dfrac{e^{-i\xi(x-x')}\sin((t-s)|\xi|)\sin((t'-s)|\xi|)}{|\xi|^{1+2H_0}},
    \]
   for almost every $(s,\xi) \in[0,t]\times\R$. Moreover,
    thanks to the fact that $|\sin(z)|\leq z$ for all $z \in
    \R$, its modulus is dominated by the integrable function
    $$\begin{cases}
    \dfrac{(t-s)(t'-s)}{|\xi|^{2\sup_n(H_{n})-1}}, &  s\in [0,t], |\xi|\leq 1, \\
    \\
    \dfrac{1}{|\xi|^{2\inf_n(H_{n})+1}}, &  s\in[0,t], |\xi|>1.
    \end{cases}$$
Then, by the dominated convergence theorem, we obtain that
\begin{align*}
\lim_{n\rightarrow \infty} \E{ \tilde{u}^{H_n}(t,x)
\tilde{u}^{H_n}(t',x') } & = c_{H_0} \int_{0}^{t} \int_{\R}
\dfrac{e^{-i\xi(x-x')}\sin((t-s)|\xi|)\sin((t'-s)|\xi|)}{|\xi|^{1+2H_0}}\,
d\xi\,ds \\
& = \E{ \tilde{u}^{H_0}(t,x) \tilde{u}^{H_0}(t',x')}.
\end{align*}

   On the other hand, in the case of the heat equation, we have
    \begin{equation}
    \label{heat covariance}
    \begin{split}
    \E{ \tilde{u}^{H_n}(t,x)
\tilde{u}^{H_n}(t',x') } & = c_{H_n} \int_{0}^{t} \int_{\R} \dfrac{
e^{-i\xi(x-x')} e^{-\frac{(t-s)\xi^2}{2}}
e^{-\frac{(t'-s)\xi^2}{2}}}{|\xi|^{2H_{n}-1}}\, d\xi\,ds.
    \end{split}
    \end{equation}
     The pointwise limit of the above integrand is given by
     \[
     \dfrac{
e^{-i\xi(x-x')} e^{-\frac{(t-s)\xi^2}{2}}
e^{-\frac{(t'-s)\xi^2}{2}}}{|\xi|^{2H_{0}-1}},
\]
 for all $s\in[0,t]$ and $\xi\in\R$, and its modulus reads
    \begin{equation*}
    \dfrac{e^{-\frac{(t+t'-2s)\xi^2}{2}}}{|\xi|^{2H_{n}-1}}.
    \end{equation*}
    Now, we use the bound
    $$e^{-ax^2}<\frac{1}{ax^2}, \quad \text{if }a>0,$$
    with $a=(t+t'-2s)/2$ (which is always positive provided that $s\in [0,t]$).
    Thus
    $$\dfrac{e^{-\frac{(t+t'-2s)\xi^2}{2}}}{|\xi|^{2H_{n}-1}}\leq \begin{cases}
    \dfrac{1}{|\xi|^{2\sup_n(H_n)-1}}, & |\xi|\leq 1, \ s\in [0,t], \\
    \\
    \dfrac{2}{(t'-t)|\xi|^{2\inf_n(H_n)+1}}, &  |\xi|>1, \  s\in [0,t].
    \end{cases}$$
    This covers all cases except $t=t'$. In this latter case,
    the modulus of the integrand appearing in
    \eqref{heat covariance} becomes
    $$\dfrac{e^{-(t-s)\xi^2}}{|\xi|^{2H_{n}-1}} \leq \begin{cases}
    \dfrac{1}{|\xi|^{2\sup_n(H_n)-1}}, & |\xi|\leq 1, \ s\in [0,t], \\
    \\
    \dfrac{\exp\Big(-(t-s)\xi^2\Big)}{|\xi|^{2\inf_n(H_n)-1}}, &  |\xi|>1, \  s\in
    [0,t],
    \end{cases}$$
    and the integrability of this function is an easy consequence of Lemma \ref{lemma: 3.1}.
    Therefore, by the dominated convergence theorem, we also obtain that
    \[
\lim_{n\rightarrow \infty} \E{ \tilde{u}^{H_n}(t,x)
\tilde{u}^{H_n}(t',x') } = \E{ \tilde{u}^{H_0}(t,x)
\tilde{u}^{H_0}(t',x')},
\]
which concludes Step 2 of the proof.

   To finish the proof of the theorem, it remains to observe that,
   since the translation by $I_0$ is clearly a continuous mapping from
   $C([0,T]\times \R)$ into itself, the convergence in distribution
   $\tilde{u}^{H_n}\xrightarrow{d}\tilde{u}^{H_0}$ implies the convergence in distribution
   $u^{H_n}\xrightarrow{d}u^{H_0}$, which was our statement.
\end{proof}


\section{Quasi-linear additive case: existence of solution}
\label{ch: quasi-linear additive}

In this section, we consider equations \eqref{eq: wave} and
\eqref{eq: heat} with a general drift coefficient $b$, where we
assume that $b:\R\to\R$ is a globally Lipschitz function. Let $T>0$.
Owing to  \eqref{def: mild solution}, we recall that a solution to
these equations is an adapted and jointly measurable process
$\{u^H(t,x), \, (t,x)\in [0,T]\times \R\}$ such that, for all
$(t,x)\in[0,T]\times \R$,
    \begin{equation}
    \begin{split}
    u^H(t,x)= & I_0(t,x)+ \int_{0}^{t}\int_{\R} G_{t-s}(x-y) W^H(ds,dy)  \\
    + &  \int_{0}^{t}\int_\R b(u^H(s,y))G_{t-s}(x-y) dy\,ds, \quad
    \mathbb{P}\text{-a.s.},
    \end{split}
    \label{eq: mild formulation quasi-linear additive}
    \end{equation}
    where the term $I_0$ and the fundamental solution $G$ are
    specified in \eqref{eq:0} and \eqref{eq:1}, respectively.

If $H>\frac12$, the existence of a unique solution to \eqref{eq:
mild formulation quasi-linear additive} follows from \cite[Thm.
4.3]{rev}, assuming that the term $I_0$ satisfies
\[
    \sup_{(t,x)\in [0,T]\times \R} |I_0(t,x)|<\infty.
\]
The case $H=\frac12$ was considered in \cite{walsh}. Finally, we
have not been able to find a proof of existence in the case
$H<\frac12$. This section is devoted to present a proof of existence
and uniqueness of solution to \eqref{eq: mild formulation
quasi-linear additive} which holds for any $H\in (0,1)$ (cf. Theorem
\ref{th: existence and uniqueness b Lipschitz}). Furthermore, we
provide sufficient conditions on the initial data ensuring that the
solution admits a H\"older-continuous version (cf. Theorem
\ref{thm:5} below).

Along this section, we will require more restrictive conditions for
the initial conditions $u_0$ and $v_0$ in the case of the wave
equation. Concretely, we consider the following assumption:

\medskip

\noindent {\bf{Hypothesis B:}} It holds that
\begin{itemize}
 \item[(a)] {\it{Wave equation}}: $u_0$ and $v_0$ are $H$-H\"older continuous and bounded.
 \item[(b)] {\it{Heat equation}}:  $u_0$ is $H$-H\"older continuous and bounded.
\end{itemize}

\medskip

Moreover, we recall that we are considering the filtration
$(\F^H_t)_{t\geq 0}$ which is generated by our fractional noise
$W^H$ (see \eqref{eq:24}).

\begin{theorem}
    \label{th: existence and uniqueness b Lipschitz}
    Let $p\geq 2$ and assume that Hypothesis B is satisfied. Then, equation
    \eqref{eq: mild formulation quasi-linear additive} has a unique solution $u^H$ in the space
    of $L^2(\Omega)$-continuous and adapted stochastic processes
    satisfying
    $$\sup_{(t,x)\in [0,T]\times\R} \E{|u^H(t,x)|^p}<\infty.$$
\end{theorem}
\begin{proof}
    We follow similar arguments as those used in \cite{extend}. We
    split the proof in four parts.

    \medskip

    \textbf{Step $\mathbf{1}$:}
    We define the following Picard iteration scheme. For $n=0$, we set
    \begin{equation}
    \label{eq: Picard iteration - 1}
    u^H_0(t,x):=I_0(t,x)+\int_{0}^{t}\int_{\R} G_{t-s}(x-y) W^H(ds,dy),
    \end{equation}
    and for $n\geq 1$ we define
    \begin{equation}
    \label{eq: Picard iteration - 2 }
    \begin{split}
    u^H_n(t,x):=u^H_0(t,x)+\int_{0}^{t} \int_\R G_{t-s}(x-y) b(u^H_{n-1}(s,y))dy\,ds .
    \end{split}
    \end{equation}
    Clearly, the process $u^H_0$ is adapted and, by step 1 in Section \ref{subsec: b unbounded}, it is $L^2(\Omega)$-continuous. Then,
    $u_0^H$ admits a jointly measurable modification (cf. \cite[Prop.
    B.1]{measure}), which will be denoted in the same way.

Owing to Lemma \ref{lemma: L^2(Omega)-continuity}, we
obtain that, for every $n\geq 0$, the Picard iteration $u^H_n$ is
$L^2(\Omega)$-continuous, and thus has a jointly measurable
modification. Moreover, by Lemma \ref{lemma: uniform boundedness in L^p(Omega)}
below, $u^H_n$ is uniformly bounded in $L^p(\Omega)$, i.e.
    $$\sup_{(t,x)\in [0,T]\times\R}\E{|u_n^H(t,x)|^p}< \infty.$$
The above two facts imply that $u^H_n$ is well-defined, for all
$n\geq 0$. On the other hand, it is clear that any Picard iteration
defines an adapted process.

    \medskip

    \textbf{Step $\mathbf{2}$:} We prove that the Picard iteration
    scheme converges in the space of $L^2(\Omega)$-continuous, adapted and $L^p(\Omega)$-uniformly bounded
    processes, which is a complete normed space when endowed with the norm
    $$||u^H||_p=\sup_{(t,x)\in [0,T]\times\R}\Big(\E{|u^H(t,x)|^p}\Big)^{1/p}.$$
    Indeed, it can be seen as the closed subset formed by adapted process of the space
    $$L^\infty([0,T]\times \R; L^p(\Omega)),$$
    which is a Banach space for any $p\geq 2$.

    Then, it is sufficient to show that the sequence of Picard iterations is Cauchy with respect
    to $||\cdot||_p$ to infer the existence of a limit.

We use that $b$ is Lipschitz and Minkowski inequality for integrals
to obtain
    \begin{equation*}
    \begin{split}
    & \Big(\E{|u^H_{n+1}(t,x)-u^H_n(t,x)|^p}\Big)^{1/p}\\
     & \quad = \Big(\E{\Big|  \int_{0}^{t}\int_{\R} G_{t-s}(x-y)[b(u^H_n(s,y))-b(u^H_{n-1}(s,y))]dy\,ds  \Big|^p}\Big)^{1/p} \\
    & \quad \leq C\Big(\E{\Big|\int_{0}^{t}\int_{\R} G_{t-s}(x-y) |u^H_{n}(s,y)-u^H_{n-1} (s,y)| dy\,ds  \Big|^p}\Big)^{1/p} \\
     & \quad \leq C \int_{0}^{t}\int_{\R}\Big(\E{G_{t-s}(x-y)^p|u^H_{n}(s,y)-u^H_{n-1} (s,y)|^p}\Big)^{1/p}dy\,ds \\
    & \quad \leq  C \int_{0}^{t}\int_{\R}G_{t-s}(x-y) \sup_{\substack{ y\in\R, \\ s'\in[0,s]}}\Big(\E{|u^H_{n}(s',y)-u^H_{n-1} (s',y)|^p}\Big)^{1/p}dy\,ds \\
    & \quad = C\int_{0}^{t} \sup_{\substack{ y\in\R, \\ s'\in[0,s]}}\Big(\E{|u^H_{n}(s',y)-u^H_{n-1}
    (s',y)|^p}\Big)^{1/p}ds.
    \end{split}
    \end{equation*}

    This inequality implies that
    \begin{equation*}
    \label{eq: bound for the sup}
    \begin{split}
    \sup_{\substack{ x\in\R, \\s\in[0,t]}} & \Big(\E{|u^H_{n+1}(s,x)-u^H_n(s,x)|^p}\Big)^{1/p}
    \leq C \int_{0}^{t} \sup_{\substack{ y\in\R, \\ s'\in[0,s]}}\Big(\E{|u^H_{n}(s',y)-u^H_{n-1} (s',y)|^p}\Big)^{1/p} ds
    \end{split}
    \end{equation*}
   If we define
    $$f_n(t):=\sup_{\substack{ x\in\R, \\s\in[0,t]}}\Big(\E{|u^H_{n+1}(s,x)-u^H_n(s,x)|^p}\Big)^{1/p},$$
    we have that
    $$f_{n}(t)\leq C\int_{0}^{t} f_{n-1}(s)ds.$$
Then, by Gr\"onwall lemma, we can conclude that $\{u^H_n\}_{n\geq
0}$ defines a Cauchy sequence in the underlying space, and therefore
it converges to a limit $u^H$, namely
\[
 \lim_{n\rightarrow \infty} \sup_{(t,x)\in [0,T]\times \R}
 \E{|u^H_n(t,x)-u^H(t,x)|^p}=0.
\]
Since any $u^H_n$ is $L^2(\Omega)$-continuous and adapted, $u^H$ has
the same properties. In particular, $L^2(\Omega)$-continuity
 implies the existence of a joint-measurable version of $u^H$.

 \medskip

    \textbf{Step $\mathbf{3}$:}  We check that the process $u^H$ is a solution of
    \eqref{eq: mild formulation quasi-linear additive}. To do this, we take $n\rightarrow
    \infty$
    with respect to the uniform $L^p(\Omega)$-norm
    in the expression
    $$u_{n+1}^H(t,x)=u_0^H(t,x)+\int_{0}^{t}\int_{\R} G_{t-s}(x-y)b(u_n^H(s,y))dy\,ds.$$
    The left-hand side, by its definition, converges to $u^H$, while for the non-constant (with respect to $n$)
    part of the right-hand side, we argue as follows:
    \begin{equation*}
    \begin{split}
    &\Big(\E{ \Big| \int_{0}^{t} \int_\R G_{t-s}(x-y) (b(u^H_{n}(s,y))-b(u^H(s,y))) dy\,ds \Big|^p }\Big)^{1/p} \\
    & \quad \leq  C \Big(\E{ \Big| \int_{0}^{t} \int_\R G_{t-s}(x-y) |u^H_{n}(s,y)-u^H(s,y)|dy\,ds \Big|^p }\Big)^{1/p} \\
    & \quad \leq  C  \int_{0}^{t}\int_{\R} G_{t-s}(x-y) \Big(\E{|u^H_{n}(s,y)-u^H(s,y)|^p}\Big)^{1/p}dy\,ds \\
    & \quad \leq C \int_{0}^{t}  \sup_{(s,y)\in [0,T]\times\R}\Big(\E{|u^H_{n}(s,y)-u^H(s,y)|^p}\Big)^{1/p}ds \\
    & \quad \leq C \sup_{(s,y)\in [0,T]\times\R} \Big(\E{|u^H_{n}(s,y)-u^H(s,y)|^p}\Big)^{1/p}.
    \end{split}
    \end{equation*}
    We note that the latter term converges to zero as $n\rightarrow
    \infty$. Thus, we have that $u^H$ satisfies \eqref{eq: mild formulation quasi-linear additive}.

    \medskip

    \textbf{Step $\mathbf{4}$:} Uniqueness can be checked by using analogous arguments as those used
    in the previous steps.
\end{proof}

\medskip

We have the following property of the sample paths of the solution
$u^H$.
\begin{theorem}\label{thm:5}
    Let $p\geq 2$. Assume that Hypothesis B is fulfilled. Let $u^H$ be the solution of
    \eqref{eq: mild formulation quasi-linear additive}. Then, for any $t,t'\in [0,T]$ and $x,x'\in \R$
    such that $|t'-t|\leq 1$ and $|x'-x|\leq 1$, the following inequalities hold true:
    \begin{equation}
    \label{eq: bound kolmogorov time}
    \sup_{x\in\R}\E{|u^H(t',x)-u^H(t,x)|^p}\leq C_p|t'-t|^{\gamma p}
    \end{equation}
    and
    \begin{equation}
    \label{eq: bound kolmogorov space}
    \sup_{t\in[0,T]}\E{|u^H(t,x')-u^H(t,x)|^p}\leq C_p|x'-x|^{Hp},
    \end{equation}
    where $\gamma=H$ for the wave equation and $\gamma=\frac H2$ for the heat equation.
    Hence, the process $u^H$ has a modification
    whose trajectories are almost surely $\gamma'$-H\"older continuous in
    time, for all $\gamma'<\gamma$, and $H'$-H\"older continuous in
    space for all $H'<H$.
\end{theorem}
\begin{proof}
    The bounds \eqref{eq: bound kolmogorov time} and \eqref{eq: bound kolmogorov space}
   are an easy corollary of the stronger results obtained in step 1 of Section \ref{subsec: b unbounded}. Indeed, in that theorem, the same kind of estimates
   have been obtained uniformly with respect to the Hurst index $H$, when restricted on a compact set
   $[a,b]\subset(0,1)$. Nevertheless, here we need to obtain \eqref{eq: bound kolmogorov time} and \eqref{eq: bound kolmogorov space}
   only for a fixed $H\in (0,1)$.
\end{proof}

\medskip

In order to conclude this section, we state and prove the two lemmas
that we used in step 1 of the proof of Theorem \ref{th: existence
and uniqueness b Lipschitz} above.

\begin{lemma}
    \label{lemma: L^2(Omega)-continuity}
    For each $n\geq 0$, the process $u^H_n$ defined by \eqref{eq: Picard iteration - 1}
    and \eqref{eq: Picard iteration - 2 } satisfies the following.
    There exists a constant $C=C(n,H)$ such that, for any $t\in
    [0,T]$ and $h\in \R$ with $t+h\leq T$, it holds
    \begin{equation}
    \sup_{x\in \R} \E{|u_n^{H}(t+h,x)-u_n^H(t,x)|^2}\leq\begin{cases}
    Ch^{\min(2H,1)}, & \text{wave equation,} \\
    Ch^{H}, &  \text{ heat equation.}
    \end{cases}
    \label{eq:6}
    \end{equation}
    and, for any $x\in \R$ and $h\in \R$ with $|h|<1$,
    \begin{equation}
    \sup_{t\in [0,T]} \E{|u_n^{H}(t,x+h)-u_n^H(t,x)|^2}\leq Ch^{2H}.
    \label{eq:7}
    \end{equation}
    In particular, the process $u^H_n$ is $L^2(\Omega)$-continuous.
\end{lemma}
\begin{proof}
    We proceed by induction. In the case $n=0$, first we study the time
    increments. We focus on the right continuity. The computations for the left continuity are
    analogous. We have
    \begin{equation*}
    \begin{split}
    & \E{|u^H_0(t+h,x)-u^H_0(t,x)|^2} \leq 2( A_1+A_2),
    \end{split}
    \end{equation*}
    where
    \begin{equation*}
    \begin{split}
    A_1  & =  |I_0(t+h,x)-I_0(t,x)|^2 \\
    A_2 & =  \mathrm{E}\Big[\Big|  \int_{0}^{t}\int_{\R} [G_{t+h-s}(x-y)-G_{t-s}(x-y)]W^H(ds,dy)
     + \int_{t}^{t+h}\int_{\R} G_{t+h-s}(x-y)W^H(ds,dy) \Big|^2
    \Big].
    \end{split}
    \end{equation*}
    In \cite{1/4<H<1/2}, Theorem 3.7, it is shown that
    $$A_1\leq \begin{cases}
    Ch^{2H} &  \text{ for the wave equation,}\\
    Ch^{H} & \text{ for the heat equation.}
    \end{cases}$$
    Concerning the term  $A_2$, we have
           $$A_2\leq 2(A_{2,1}+A_{2,2}),$$
    where
    \begin{equation*}
    \begin{split}
    A_{2,1} & =  \E{\Big|\int_{0}^{t}\int_{\R} [G_{t+h-s}(x-y)-G_{t-s}(x-y)]W^H(ds,dy)\Big|^2 }, \\
    A_{2,2} & = \E{\Big|\int_{t}^{t+h}\int_{\R} G_{t+h-s}(x-y)W^H(ds,dy) \Big|^2 }.
    \end{split}
    \end{equation*}

These terms have been studied in the proof of Theorem 2.8,
concretely $A_{2,1}$ corresponds to term $J_1$ in that theorem and
term $A_{2,2}$ corresponds to $I_1$. So,
\begin{equation*}
    A_{2,1}  \leq  \begin{cases}
    Ch^{1+2H}, & \text{for the wave equation}, \\
    Ch^{\frac12+H}, & \text{for the heat equation},
    \end{cases}
    \end{equation*}
and
    \begin{equation*}
    A_{2,2}  \leq  \begin{cases}
    Ch^{1+2H}, & \text{for the wave equation}, \\
    Ch^{\frac12+H}, & \text{for the heat equation}.
    \end{cases}
    \end{equation*}
    Putting together the above estimates, we obtain the validity of
    \eqref{eq:6} for $n=0$.

    Regarding the space increments, we have, for any $h\in \R$ with
    $|h|<1$,
    $$\E{|u_0^H(t,x+h)-u_0^H(t,x)|^2} \leq 2(B_1+B_2),$$
    where
    \begin{equation*}
    \begin{split}
    B_1 & = |I_0(t,x+h)-I_0(t,x)|^2,\\
    B_2 & =  \E{\Big|  \int_{0}^{t} \int_\R [G_{t-s}(x+h-y)-G_{t-s}(x-y)] W^H(ds,dy)
    \Big|^2}.
    \end{split}
    \end{equation*}
    As before, by \cite[Thm. 3.7]{1/4<H<1/2}, we have
    $$B_1\leq Ch^{2H}$$
    for both heat and wave equations. The term  $B_2$ corresponds to
    $J_2$ in the proof of Theorem 2.8, hence
    \begin{equation*}
    B_2  \leq C|h|^{1+2H}.
    \end{equation*}
  So, we have proved \eqref{eq:7} for $n=0$.

\medskip

    We suppose now by induction hypothesis that $u^H_n$ satisfies
    \eqref{eq:6} and \eqref{eq:7}. Let us compute the time increments of
    $u^H_{n+1}$, for  $0<h<<1$:
    \begin{equation*}
    \begin{split}
    \E{|u_{n+1}^H(t+h,x)- & u_{n+1}^H(t,x)|^2} \leq  3( D_1+D_2+D_3),
    \end{split}
    \end{equation*}
    where
    \begin{equation*}
    \begin{split}
    D_1 & = \E{|u_{0}^H(t+h,x)-u_{0}^H(t,x)|^2}, \\
    D_2 & = \E{ \Big(\int_{0}^{t}\int_{\R} G_s(y) |b(u_n^H(t+h-s,x-y)) -b(u_n^H(t-s,y)|\, dy\,ds\Big)^2 },\\
    D_3 & = \E{\Big(\int_{t}^{t+h}\int_{\R} G_s(y) |b(u_n(t+h-s,x-y))|\, dy\,ds \Big)^2}.
    \end{split}
    \end{equation*}
    We already showed that $D_1$ is bounded as the right hand side of \eqref{eq:6}, so we only need to handle $D_2$ and $D_3$.
     As in Lemma 19 of \cite{extend}, first we compute $D_2$.
    Namely, using that $b$ is Lipschitz and applying Cauchy-Schwarz inequality and Fubini theorem, we have
    \begin{equation*}
    \begin{split}
    D_2 & \leq C\, \Big( \int_{0}^{t}\int_{\R} G_s(y)dy\,ds \Big) \mathrm{E}\Big[\int_{0}^{t}\int_{\R}G_{s}(y) |u_n^H(t+h-s,x-y)-u^H_n(t-s,x-y)|^2\, dy\,ds\Big]\\
     & \leq C\, \E{  \int_{0}^{t}\int_{\R} G_{s}(y) |u_n^H(t+h-s,x-y)-u^H_n(t-s,x-y)|^2 \, dy\,ds } \\
     &  =  C \int_{0}^{t}\int_{\R} G_{s}(y) \E{|u_n^H(t+h-s,x-y)-u^H_n(t-s,x-y)|^2}\, dy\,ds \\
    & \leq   \begin{cases}
    Ch^{2H}, & \text{wave equation,} \\
    Ch^{H}, &  \text{ heat equation.}
    \end{cases}
    \end{split}
    \end{equation*}
    Notice that in the last inequality we used the induction hypothesis.

    Regarding $D_3$, we have
    \begin{equation*}
    D_3   \leq  C\int_{t}^{t+h}\int_{\R} \Big(1+\E{|u_n^H(t+h-s,x-y)|^2}\Big) G_s(y)dy\,ds.
    \end{equation*}
    The uniform boundedness in $L^2(\Omega)$ of $u^H_n$ (by Lemma \ref{lemma: uniform boundedness in L^p(Omega)})
     gives that
    $$D_3\leq  C\int_{t}^{t+h}\int_{\R} G_s(y)dy\,ds\leq C h,$$
    for both wave and heat equations.
    Thus, taking into account the above estimates for $J_1$, $J_2$ and $J_3$,
    we obtain that $u^H_{n+1}$ satisfies \eqref{eq:6}.

    We are left to deal with the spatial increments of $u^H_{n+1}$.
    Indeed, we have
    $$\E{|u^H_{n+1}(t,x+h)-u^H_{n+1}(t,x)|^2}\leq 2(K_1+K_2),$$
    where
    \begin{equation*}
    \begin{split}
    K_1  & = \E{|u^H_0(t,x+h)-u^H_0(t,x)|^2},\\
    K_2  &  = \E{\Big(\int_{0}^{t} \int_{\R} |b(u^H_n(t-s,x+h-y))-b(u^H_n(t-s,x-y))|G_s(y)dy\,ds\Big)^2 }.
    \end{split}
    \end{equation*}
    The term $K_1$ has already been studied, and $K_2$ can be treated as the term
    $J_2$, obtaining that $K_2\leq C |h|^{2H}$. So we can infer that
    \eqref{eq:7} is fulfilled for $u^H_{n+1}$.
\end{proof}

\medskip

\begin{lemma}
    \label{lemma: uniform boundedness in L^p(Omega)}
    Let $p\geq 2$ and $[a,b]\subset (0,1)$.
    Let $u^H_n$, $n\geq 0$, be the Picard iteration scheme defined in
    \eqref{eq: Picard iteration - 1} and \eqref{eq: Picard iteration - 2 }.
    Then,
    $$\sup_{n\geq 0}\sup_{H\in[a,b]}\sup_{(t,x)\in [0,T]\times R} \E{|u_n^H(t,x)|^p}<\infty.$$
\end{lemma}
\begin{proof}
    First, we have
    \begin{equation*}
    \begin{split}
    \E{|u^H_0(t,x)|^p}\leq C_p\Big(|I_0(t,x)|^p+ \E{\Big| \int_0^t\int_\R G_{t-s}(x-y) W^H(ds,dy)\Big|^p}\Big).
    \end{split}
    \end{equation*}
    By \cite{rev}, Lemma 4.2, we have that
    $$\sup_{(t,x)\in [0,T]\times\R} |I_0(t,x)|<\infty,$$
    and this is uniform in $H$, since we are considering the same initial conditions for every $H$.
    Regarding the stochastic term, arguing as in \eqref{eq: estimate burkholder
    type} and applying Lemma \ref{lemma: 3.1}, we get
    \begin{equation*}
    \begin{split}
    \E{\Big| \int_0^t\int_\R G_{t-s}(x-y) W^H(ds,dy)\Big|^p}  & =
     z_p c_H^{p/2} \Big[  \int_{0}^{t}\int_{\R} |\F G_{t-s}(x-\cdot)(\xi)|^2|\xi|^{1-2H}d\xi\,ds  \Big]^{p/2} \\
    & \leq \begin{cases}
    C_{p}\, \Big(t^{1+2H}\Big)^{p/2}, & \text{wave equation,} \\
    C_{p}\, \Big(t^{H}\Big)^{p/2}, & \text{heat equation.}
    \end{cases}
    \end{split}
    \end{equation*}
    The last inequality comes from an estimate essentially identical to the one already computed
    in \eqref{eq: estimate for I_1}.  All above constants which are dependent on $H$
    can be uniformly bounded, provided that $H$ is in the compact interval
      $[a,b]\subset (0,1)$.
    The above considerations yield
    $$\sup_{H\in[a,b]}\sup_{(t,x)\in [0,T]\times\R}\E{|u^H_0(t,x)|^p}<\infty.$$

Next, owing to \eqref{eq: Picard iteration - 2 } we can infer that
    \begin{equation*}
    \label{eq: start of the estimate for uniform L^p boundedness}
    \begin{split}
    \E{|u^H_{n+1}(t,x)|^p}  \leq C \Big( 1 +
    \E{\Big| \int_{0}^{t} \int_\R G_{t-s}(x-y) b(u^H_{n}(s,y))dy\,ds\Big|^p}\Big).
    \end{split}
    \end{equation*}
    If we apply H\"older inequality, we obtain
    \begin{equation}
    \label{eq: use of Holder inequality}
    \begin{split}
    & \E{\Big| \int_{0}^{t} \int_\R G_{t-s}(x-y) b(u^H_{n}(s,y))dy\,ds\Big|^p} \\
    & \qquad  \leq C \E{\int_{0}^{t} \int_\R G_{t-s}(x-y) \Big(1+|u^H_{n}(s,y))|^p\Big) dy\,ds } \\
    & \qquad = C_{1} + C_{2} \int_{0}^{t} \int_\R G_{t-s}(x-y)\E{|u^H_{n}(s,y))|^p} dy\,ds \\
    & \qquad \leq C_{1}+  C_{2} \int_{0}^{t} \int_\R \sup_{H\in[a,b]}\sup_{(s',y)\in [0,s]\times\R} \E{|u^H_{n}(s',y))|^p} G_{t-s}(x-y)dy\,ds \\
    & \qquad \leq C_{1}+  C_{2} \int_{0}^{t} \sup_{H\in[a,b]}\sup_{(s',y)\in [0,s]\times\R} \E{|u^H_{n}(s',y))|^p} ds.\\
    \end{split}
    \end{equation}
    The constants appearing in the previous calculations are clearly independent of $H$. Then,
    we have
    \begin{equation*}
    \begin{split}
    \sup_{H\in[a,b]}\sup_{(t',y)\in [0,t]\times\R}&\E{|u^H_{n+1}(t',y)|^p} \\
    & \leq C_{1}+  C_{2} \int_{0}^{t} \sup_{H\in[a,b]} \sup_{(s',y)\in [0,s]\times\R} \E{|u^H_{n}(s',y))|^p} ds.\\
    \end{split}
    \end{equation*}
    We conclude the proof by applying Gr\"onwall lemma.
\end{proof}


\section{Quasi-linear additive case: weak convergence}
\label{sec:1}

This section is devoted to prove that the mild solution $u^H$ of
equation \eqref{eq: wave} (resp. \eqref{eq: heat}) converges in law
in the space of continuous functions, as $H\to H_0$, to the solution
$u^{H_0}$ of \eqref{eq: wave} (resp. \eqref{eq: heat}) corresponding
to the Hurst index $H_0$.

Throughout this section, we fix $H_0\in (0,1)$ and any sequence
$(H_n)_{n\geq 1}$ converging to $H_0$. Then, we consider the
following assumptions for the initial data: 

\medskip

\noindent {\bf{Hypothesis C:}} For some $\alpha>H_0$, it holds that
\begin{itemize}
 \item[(a)] {\it{Wave equation}}: $u_0$ and $v_0$ are $\alpha$-H\"older continuous and bounded.
 \item[(b)] {\it{Heat equation}}:  $u_0$ is $\alpha$-H\"older continuous and bounded.
\end{itemize}

\medskip

Without any loss of generality, we assume that $H_n\leq \alpha$, for all $n\geq 1$. 
Hence, we will be able to apply the results of
the previous section for all these Hurst indexes.

\medskip

The main strategy to prove that $u^{H_n}$ converges in law to
$u^{H_0}$ can be summarized as follows. Recall that $b$ is assumed
to be globally Lipschitz. Let $\eta$ be a deterministic function in
$C([0,T]\times \R)$, and consider the (deterministic) integral
equation
\begin{equation}
\label{eq: integral equation}
z(t,x)= \int_{0}^{t}\int_{\R} b(z(s,y))G_{t-s}(x-y)dsdy + \eta(t,x),
\end{equation}
which is defined on the space $C([0,T]\times \R)$, endowed with the
metric of uniform convergence on compact sets.

We will prove that  \eqref{eq: integral equation} admits a unique
solution. This allows us to define the solution operator
\begin{equation}
F:C([0,T]\times \R)\longrightarrow C([0,T]\times \R) \label{eq:10}
\end{equation}
by $(F\eta)(t,x):=z(t,x)$. We will show that this operator is
continuous. Note that $u^{H_n}=F(\bar u^{H_n})$ (almost surely), for
all $n\geq 0$, where $\bar u^{H_n}$ denotes the solution in the
linear additive case (i.e. $b=0$). Moreover, by Theorem \ref{th:
convergence linear additive case}, $\bar u^{H_n}$ converges in law,
in the space of continuous functions, to $\bar u^{H_0}$. Therefore,
we can apply Theorem 2.7 of \cite{billingsley} to obtain
the desired result.

\medskip

Here is the main result of the paper.
\begin{theorem}\label{thm:9}
Assume that Hypothesis C is fulfilled and $b$ is globally Lipschitz.
Then, $u^{H_n}\xrightarrow{d}u^{H_0}$, as $n\to\infty$, where the
convergence holds in distribution in the space $C([0,T]\times \R)$.
\end{theorem}

\medskip

The proof of the above theorem will be tackled in the following
three subsections. Indeed, we need to distinguish the case of the
wave equation from the one of the heat equation. Moreover, for the
heat equation, we split the analysis in two subcases: bounded $b$
and possibly unbounded $b$. As it will be made clear in the sequel,
in the latter case, the above-explained strategy based on the {\it
solution operator} cannot be applied, so the case $b$ unbounded will
be studied separately.


\subsection{Wave equation}
\label{sec:w}

In this section, we provide the proof of Theorem \ref{thm:9} for the
stochastic wave equation \eqref{eq: wave}. For this, as already
explained, it suffices to prove that equation \eqref{eq: integral
equation} has a unique solution and that the solution operator
\eqref{eq:10} is continuous. These two facts will be proved in
Theorem \ref{thm:19} below.

We recall that the fundamental solution $G$ of the wave equation on $[0,\infty)\times \R$ is
$$G_t(x)=\frac{1}{2}1_{\{|x|\leq t\}}.$$
We will make use of the following ad hoc version of Gr\"onwall lemma
(\cite{giulio}). We give its proof for the sake of completeness.
\begin{lemma}
    \label{lemma: Gronwall wave}
    Let $\{f_n, \, n\geq 0\}$ be a sequence of real-valued non-negative functions defined
    on $[0,T]\times[a-T,b+T]$, for some $a,b\in \R$ such that $a<b$, and $T>0$. Suppose that there exist
    $\lambda,\mu>0$ such that, for every $(t,x)\in [0,T]\times [a,b]$  and $n\geq 0$,
    $$f_{n+1}(t,x)\leq \lambda+\frac{\mu}{2}\int_{0}^{t}\int_{x-t+s}^{x+t-s} f_n(s,y)\, dyds,$$
    and that $f_0$ is bounded. Then, for every $n\geq 0$ and $(t,x)\in[0,T]\times[a,b]$, it holds that
    \begin{equation}
    \label{eq: inequality in the lemma}
    f_{n}(t,x)\leq \lambda \sum_{k=0}^{n-1} \frac{(\mu t^2)^k}{k!} + ||f_0||_\infty\frac{(\mu t^2)^n}{n!},
    \end{equation}
    which in particular implies that
    $$\limsup_{n\to \infty} f_n(t,x)\leq \lambda\exp(\mu t^2).$$
\end{lemma}
\begin{proof}
    We prove it by induction: the case $n=1$ reduces to the  inequality
    $$f_1(t,x)\leq \lambda+\mu t^2||f_0||_\infty,$$
    that is clearly satisfied.
    We go on with the inductive step: if \eqref{eq: inequality in the lemma} holds true, then
    \begin{equation*}
    \begin{split}
    f_{n+1}(t,x) & \leq \lambda+\frac{\mu}{2}\int_{0}^{t}\int_{x-t+s}^{x+t-s} \Big[\lambda\sum_{k=0}^{n-1}\frac{(\mu s^2)^k}{k!}+||f_0||_\infty\frac{(\mu s^2)^n}{n!}\Big] dsdy \\
    & = \lambda+\frac{\mu}{2}\int_{0}^{t}2(t-s)\Big[\lambda\sum_{k=0}^{n-1}\frac{(\mu s^2)^k}{k!}+||f_0||_\infty\frac{(\mu s^2)^n}{n!}\Big] ds \\
    & \leq \lambda+\mu\int_{0}^{t} t \Big[\lambda\sum_{k=0}^{n-1}\frac{(\mu s^2)^k}{k!}+||f_0||_\infty\frac{(\mu s^2)^n}{n!}\Big] ds \\
    & = \lambda+\mu\Big[\lambda\sum_{k=0}^{n-1}\frac{\mu^k (t^2)^{k+1}}{k!(2k+1)}+||f_0||_\infty\frac{\mu^n (t^2)^{n+1}}{n!(2n+1)}\Big] \\
    & =  \lambda+\lambda\sum_{k=0}^{n-1}\frac{\mu^{k+1} (t^2)^{k+1}}{k!(2k+1)}+||f_0||_\infty\frac{\mu^{n+1} (t^2)^{n+1}}{n!(2n+1)}\\
    & \leq  \lambda \sum_{k=0}^{n}\frac{\mu^{k} (t^2)^{k}}{k!}+||f_0||_\infty\frac{\mu^{n+1} (t^2)^{n+1}}{(n+1)!}, \\
    \end{split}
    \end{equation*}
    which is our thesis. In the last two inequalities, we shifted by one the index of the sum
    and we used the fact that $4k^2+6k+2> k+1$, for every $k\in \N$. If we take the $\limsup$ as
    $n\to \infty$ in both sides of the inequality we also obtain easily that
    $$\limsup_{n\to \infty} f_n(t,x)\leq \lambda\exp(\mu t^2).$$
\end{proof}

\medskip

We will use the above Gr\"onwall-type  lemma to prove the following
theorem, proved also in \cite{giulio}.
\begin{theorem}\label{thm:19}
    Let $\eta\in C([0,T]\times \R)$ and consider the deterministic equation
    \eqref{eq: integral equation} in the case where $G$ is the fundamental solution of the wave equation.
    Then, \eqref{eq: integral equation} has a unique solution $z\in C([0,T]\times \R)$.
    Moreover, the solution operator
    $$F:C([0,T]\times \R)\to C([0,T]\times \R)$$
    defined by $F(\eta)=z$ is continuous, if we endow $C([0,T]\times \R)$
    with the metric of uniform convergence on compact sets.
\end{theorem}
\begin{proof}
    We define the Picard iteration scheme
    \begin{equation}
    \begin{split}
    z_0(t,x) & :=\eta(t,x)  \\
    z_n(t,x) & :=\int_{0}^{t}\int_{\R} G_{t-s}(x-y) b(z_{n-1}(s,y))dyds+\eta(t,x)  \\
    &   =\frac{1}{2}\int_{0}^{t}\int_{x-t+s}^{x+t-s} b(z_{n-1}(s,y))dyds+\eta(t,x), \quad n\geq 1. \\
    \end{split}
    \label{eq:12}
    \end{equation}
    Clearly, the above expressions of the Picard scheme are well-defined. Moreover, since $b$ is Lipschitz
    continuous, if $z_{n-1}$ is continuous then also $b\circ z_{n-1}$ is so. This gives by induction that $z_n$ is a continuous function.
    Moreover, we will show that $z_n$ converges uniformly on compact sets on $[0,T]\times \R$.
    More precisely, we prove that the sequence $\{z_n\}_{n\geq 0}$ is uniformly Cauchy on $[0,T]\times[-L,L]$,
    for every $L>0$. Indeed, for all $(t,x)\in [0,T]\times [-L,L]$, we have
    \begin{equation*}
    \begin{split}
    |z_{n+1}(t,x)-z_{n}(t,x)| & = \left|\frac{1}{2}\int_{0}^{t} \int_{x-t+s}^{x+t-s} [b(z_{n}(s,y))-b(z_{n-1}(s,y))] dy\,ds \right|\\
     & \leq C\int_{0}^{t} \int_{x-t+s}^{x+t-s} \Big|z_{n}(s,y)-z_{n-1}(s,y)\Big| dy\,ds.
    \end{split}
    \end{equation*}
    We can apply Lemma \ref{lemma: Gronwall wave} to the sequence of functions
    $f_n:=|z_{n+1}-z_{n}|$ and with $\lambda=0$ and $\mu=2C$, obtaining that
    \begin{equation*}
    \begin{split}
    |z_{n+1}(t,x)-z_{n}(t,x)| & \leq \Big(\sup_{(s,y)\in [0,T]\times[-L-T,L+T]}|z_1(s,y)-z_0(s,y)|\Big) \frac{(2C t^2)^n}{n!} \\
     & \leq \Big(\sup_{(s,y)\in [0,T]\times[-L-T,L+T]}|z_1(s,y)-z_0(s,y)|\Big)\frac{(2C
    L^2)^n}{n!}.
    \end{split}
    \end{equation*}
    Notice that the latter bound does not depend on $t$ and $x$. This remark,
    together with the fact that the function $z_1-z_0$ is bounded on any compact set, and
    that the sum
    $\sum_{k=0}^{\infty} \frac{(2C L^2)^n}{n!}$
    is convergent, yield that the sequence $\{z_n(t,x)\}_{n\geq 0}$ is uniformly Cauchy on
    $[0,T]\times [-L,L]$.
    Let $z(t,x)$ denote its limit. Then, by the uniqueness of the pointwise limit,
    the fact that $C([0,T]\times \R)$ is a complete metric space
    (with the underlying metric) and that $z_n$, $n\geq 0$, are continuous functions, we have that
     $z$ is also a continuous function in $C([0,T]\times \R)$.

    Letting $n\rightarrow \infty$ in \eqref{eq:12} and observing that $b\circ z_n\to b \circ z$
    uniformly on compact sets, one easily gets that $z$ solves equation \eqref{eq: integral equation}.

    The uniqueness of the solution comes from a simple remark:
    suppose we have two solutions $z_1,z_2$ relative to the same $\eta$.
    Then, for a fixed $L>0$ and for any $(t,x)\in [0,T]\times[-L,L]$, we have
    \begin{equation*}
    \begin{split}
    |z_1(t,x)-z_2(t,x)|  & \leq  \frac{1}{2}\int_{0}^{t}\int_{x-t+s}^{x+t-s} |b(z_1(s,y))-b(z_2(s,y))|dy\,ds \\
     & \leq  C\int_{0}^{t}\int_{x-t+s}^{x+t-s} |z_1(s,y)-z_2(s,y)|dy\,ds. \\
    \end{split}
    \end{equation*}
    It remains to apply Lemma \ref{lemma: Gronwall wave}
    to obtain the uniqueness for every $L>0$, and thus for the equation on the whole space.

    \medskip

    Let us now turn to the analysis of the solution operator
    $F:C([0,T]\times \R)\longrightarrow C([0,T]\times \R)$, which is defined by
    $F(\eta)(t,x):=z(t,x)$. We need to prove that this operator is continuous with respect to the metric of
    uniform convergence on compact sets. That is, we show the continuity of the restricted mapping
    $$F_L:C([0,T]\times \R)\longrightarrow C([0,T]\times \R),$$
    for every  $L>0$.

    We denote by  $||\cdot||_{\infty,L}$ the supremum norm on $C([0,T]\times [-L,L])$.
    Let $z_1:=F(\eta_1)$ and $z_2:=F(\eta_2)$ for some $\eta_1, \eta_2 \in C([0,T]\times \R)$.
    Then, for $(t,x)\in [0,T]\times[-L,L]$,
    \begin{equation*}
    \begin{split}
    |z_1(t,x)-  z_2(t,x)|  & \leq \int_{0}^{t}\int_{x-t+s}^{x+t-s} |b(z_1(s,y))-b(z_2(s,y))|dy\,ds +
    |\eta_1(t,x)-\eta_2(t,x)| \\
     & \leq  C \int_{0}^{t}\int_{x-t+s}^{x+t-s} |z_1(s,y)-z_2(s,y)|dy\,ds + ||\eta_1-\eta_2||_{\infty,L}.
    \end{split}
    \end{equation*}
    Here, we apply again Lemma \ref{lemma: Gronwall wave} to obtain that
     $$||z_1-z_2||_{\infty,L}\leq C ||\eta_1-\eta_2||_{\infty,L}.$$
\end{proof}


\subsection{Heat equation: $b$ bounded}
\label{subsec: heat b bounded}

In this section, we prove Theorem \ref{thm:9} for the stochastic heat equation \eqref{eq: heat}
in the particular case where the drift $b$ is assumed to be a bounded function.
This is necessary in order to construct a Picard iteration scheme to solve
equation \eqref{eq: integral equation},

Recall that the fundamental solution of the heat equation in
$[0,\infty)\times \R$ is given by
$$G_t(x)=\frac{1}{\sqrt{2\pi t}}e^{-\frac{|x|^2}{2t}}.$$

As we did in the previous subsection, first we establish an ad hoc
version of Gr\"onwall lemma.

\begin{lemma}
    \label{lemma: gronwall heat}
    Let $\{f_n\}_{n\geq 1}$, $f_n:[0,T]\times \R\to \R$, be a sequence of functions that satisfy,
    for every $(t,x)\in [0,T]\times \R$, the following inequality: for some $\mu,\lambda>0$,
    \begin{equation*}
    \begin{split}
    |f_{n+1}(t,x)-  f_n(t,x)|
     \leq  \mu\int_{0}^{t} \int_{\R} \frac{1}{\sqrt{2\pi (t-s)}} e^{-\frac{|x-y|^2}{2(t-s)}} |b(f_n(s,y))-b(f_{n-1}(s,y))|dy\,ds + \lambda,
    \end{split}
    \end{equation*}
    where $b:\R\to \R$ is bounded and Lipschitz continuous with Lipschitz constant $C$. Then, we have that, for any $n\geq 1$ and
    $(t,x)\in[0,T]\times \R$,
    \begin{equation*}
    \label{eq: estimate gronwall heat}
    |f_{n+1}(t,x)-f_{n}(t,x)|\leq  2||b||_\infty \frac{C^{n-1}(\mu t)^{n}}{n!}+\sum_{k=0}^{n-1}\frac{\lambda t^{k}}{k!}.
    \end{equation*}
    As a consequence, we also have that
    $$\limsup_{n\to \infty} \Big(\sup_{x\in\R} |f_{n+1}(t,x)-f_n(t,x)|\Big) \leq \lambda e^t.$$
\end{lemma}
\begin{proof}
    We prove it by induction. First, we compute
    \begin{equation*}
    \begin{split}
    |f_2(t,x)-f_1(t,x)|
     & \leq \mu\int_{0}^{t} \int_{\R} \frac{1}{\sqrt{2\pi (t-s)}} e^{-\frac{|x-y|^2}{2(t-s)}} |b(f_1(s,y))-b(f_{0}(s,y))|dy\,ds+\lambda \\
     &\leq  2\mu||b||_\infty \int_{0}^{t} \int_{\R} \frac{1}{\sqrt{2\pi (t-s)}} e^{-\frac{|x-y|^2}{2(t-s)}}dy\,ds +\lambda \\
     & \leq 2\mu||b||_\infty \int_{0}^{t} 1 ds +\lambda\\
     & = 2\mu t ||b||_\infty+\lambda.
    \end{split}
    \end{equation*}
    For the inductive step, we have to exploit the Lipschitz continuity of $b$:
    \begin{equation*}
    \begin{split}
    |f_{n+1}(t,x)-f_n(t,x)|
     & \leq\mu\int_{0}^{t} \int_{\R} \frac{1}{\sqrt{2\pi (t-s)}} e^{-\frac{|x-y|^2}{2(t-s)}} |b(f_n(s,y))-b(f_{n-1}(s,y))|dy\,ds+\lambda \\
     & \leq \mu C\int_{0}^{t} \int_{\R} \frac{1}{\sqrt{2\pi (t-s)}} e^{-\frac{|x-y|^2}{2(t-s)}} |f_n(s,y)-f_{n-1}(s,y)|dy\,ds+\lambda \\
     & \leq \mu C\int_{0}^{t} \int_{\R} \frac{1}{\sqrt{2\pi (t-s)}} e^{-\frac{|x-y|^2}{2(t-s)}}\Big[  2
     ||b||_\infty \frac{C^{n-2}(\mu s)^{n-1}}{(n-1)!}
     +\sum_{k=0}^{n-2}\frac{\lambda s^{k}}{k!} \Big]dy\,ds + \lambda\\
     & =\int_{0}^{t} \Big[2 ||b||_\infty  \frac{\mu^n C^{n-1}s^{n-1}}{(n-1)!}+ \sum_{k=0}^{n-2}\frac{\lambda s^k}{k!}\Big]dy\,ds +\lambda \\
     & =2 ||b||_\infty C^{n-1} \frac{(\mu t)^n}{n!}+\sum_{k=1}^{n-1}\frac{\lambda t^k}{k!}+\lambda.
     \end{split}
    \end{equation*}
    A direct consequence of this fact is that
    $$\limsup_{n\to \infty} |f_{n+1}(t,x)-f_n(t,x)| \leq \lambda e^t,$$
    which concludes the proof.
\end{proof}

\medskip

The proof of Theorem \ref{thm:9} in our standing case follows from the following result.

\begin{theorem}
    \label{th: solution deterministic eq heat equation}
    Let $\eta\in C([0,T]\times \R)$ and consider the deterministic equation
    \eqref{eq: integral equation} in the case where $G$ is the fundamental solution of the heat equation, and
    such that $b$ is Lipschitz and bounded.
    Then, \eqref{eq: integral equation} has a unique solution $z\in C([0,T]\times \R)$.
    Moreover, the solution operator
    $$F:C([0,T]\times \R)\to C([0,T]\times \R)$$
    defined by $F(\eta)=z$ is continuous, if we endow $C([0,T]\times \R)$
    with the metric of uniform convergence on compact sets.
\end{theorem}
\begin{proof}
    As in the case of the wave equation, we consider the Picard iteration scheme
    \begin{equation*}
    \begin{split}
    z_0(t,x) & =\eta(t,x)  \\
    z_n(t,x) & =\int_{0}^{t}\int_{\R} G_{t-s}(x-y) b(z_{n-1}(s,y))dyds+\eta(t,x)  \\
    &  =\int_{0}^{t}\int_\R \frac{1}{\sqrt{2\pi (t-s)}}e^{-\frac{|x-y|^2}{2(t-s)}} b(z_{n-1}(s,y))dyds+\eta(t,x),
    \quad n\geq 1.
    \end{split}
    \end{equation*}
    We clearly have that $z_0$ is continuous. Assume that $z_{n-1}$ is continuous, and we check that
    $z_n$ is so. In fact, let $(t,x)\in [0,T]\times \R$ and pick a sequence $(t_m,x_m)\to (t,x)$
    as $m\to\infty$. Then,
    \begin{equation*}
    \begin{split}
    z_n(t_m,x_m) & = \int_{0}^{t_m } \int_{\R} G_{t_m-s}(x_m-y)b(z_n(s,y))dy\,ds+ \eta(t_m,x_m) \\
    & =  \int_{0}^{t_m}\int_{\R}G_{s'}(y')b(z_{n-1}(t_m-s',x_m-y'))dy'\,ds'+\eta(t_m,x_m) \\
    & =  \int_{0}^{\sup_m t_m}\int_{\R}   1_{[0,t_m]\times\R}(s',y')G_{s'}(y')b(z_{n-1}(t_m-s',x_m-y'))dy'\,ds'+\eta(t_m,x_m).
    \end{split}
    \end{equation*}
    Thanks to the continuity of $b$ and $z_{n-1}$, the latter integrand converges point-wise to
    $$1_{[0,t]\times \R}(s',y')G_{s'}(y')b(z_{n-1}(t-s',x-y')).$$
    Since $b$ is bounded and $G$ has finite integral over $[0,\sup_m t_m]\times \R$, we can apply the dominated
    convergence theorem to obtain that
    \begin{equation*}
    \begin{split}
    \lim_{m\rightarrow \infty} z_n(t_m,x_m) = z_n(t,x),
    \end{split}
    \end{equation*}
    so $z_n$ is continuous.

  For every $(t,x)\in[0,T]\times\R$, we can infer that
     \begin{equation*}
    \begin{split}
    |z_{n+1}(t,x) -z_{n}(t,x)|
      \leq \int_{0}^{t} \int_{\R} \frac{1}{\sqrt{2\pi (t-s)}} e^{-\frac{|x-y|^2}{2(t-s)}}
     |b(z_n(s,y))-b(z_{n-1}(s,y))|dy\,ds.
     \end{split}
    \end{equation*}
    By Lemma \ref{lemma: gronwall heat}, we get
    \[
     |z_{n+1}(t,x) -z_{n}(t,x)| \leq  2 ||b||_\infty \frac{C^{n-1}t^n}{n!}
      \leq 2 ||b||_\infty \frac{C^{n-1}T^n}{n!}.
    \]
    Since the rightmost term of this inequality is the general term of a converging series,
    and the series does not depend on $(t,x)$, we can infer that the sequence
    $\{z_n(t,x)\}_{n\geq 0}$ is uniformly Cauchy in $[0,T]\times \R$.
    This means that a limit $z$ exists and, since  $z_n\to z$ uniformly, $z\in C([0,T]\times\R)$.
    Moreover, it is straightforward to verify that $z$ is the solution to equation
    \eqref{eq: integral equation}.
 Finally, uniqueness of solution can be easily checked
      by applying again Lemma \ref{lemma: gronwall heat}.

    \medskip

    As far as the continuity of the solution operator $F:C([0,T]\times \R)\to C([0,T]\times \R)$
    is concerned, where
    $F(\eta)(t,x)=z(t,x)$, this property can be verified similarly to the case of the wave equation, but
    applying Lemma \ref{lemma: gronwall heat}.
\end{proof}


\subsection{Heat equation: $b$ general}
\label{subsec: b unbounded}

In this section, we aim to verify the validity of Theorem
\ref{thm:9} for the stochastic heat equation \eqref{eq: heat} in the
case of a general globally Lipschitz coefficient $b$. Recall that
the initial condition $u_0$ is assumed to satisfy Hypothesis C. In
particular, $u_0$ is $\alpha$-H\"older continuous for some
$\alpha>H_0$.

We will use a truncation argument on the drift $b$: for every $m\geq 1$, set
$$b_m(x):=\begin{cases}  b(x)\wedge m, & \text{if }b(x)\geq0, \\
b(x )\vee -m, & \text{if }b(x)<0.
\end{cases}$$
We have that $b_m$ is bounded and Lipschitz continuous, and converge
pointwise to $b$, as $m\rightarrow \infty$. Moreover, a unique
Lipschitz constant can be fixed for all functions $b_m$, $m\geq 1$,
and $b$. We define $u_m^{H_n}$ to be the solution of \eqref{eq: mild
formulation quasi-linear additive} where $b$ is replaced by $b_m$,
and corresponding to the Hurst index $H_n$. An immediate consequence
of Section \ref{subsec: heat b bounded} is that, for any $m\geq 1$,
\begin{equation}
u_m^{H_n}\xrightarrow[n\to\infty]{d} u_m^{H_0}
\label{eq:20}
\end{equation}
 on
$C([0,T]\times \R)$.

\medskip

Then, the proof of Theorem \ref{thm:9} is split in three steps.

\medskip

\textbf{Step $\mathbf{1}$:}
    First, we check that the family of laws of $\{u^{H_n}\}_{n\geq 1}$ is tight in
    $C([0,T]\times \R)$. For this, we will apply the
criterion stated in Theorem \ref{th: centsov}. We point out that,
indeed,
    the computations of this step are
    valid for both heat and wave equations.

    Notice that condition (i) of Theorem \ref{th: centsov} is clearly satisfied, since
    $u^{H_n}(0,0)$ is deterministic and does not depend on $n$. Regarding condition (ii),
    let $t,t'\in [0,T]$ and $x,x'\in \R$ with $t'\geq t$ and $x'\geq x$, and we can suppose that $|x-x'|<1$ and $|t-t'|<1$. We aim to estimate
    \begin{equation}
    \begin{split}
    \E{|u^{H_n}(t',x')  -u^{H_n}(t,x)|^p}
     & \leq  C_p\Big(  \E{|u^{H_n}(t',x')  -u^{H_n}(t,x')|^p}+\E{|u^{H_n}(t,x')  -u^{H_n}(t,x)|^p}    \Big) \\
     & =: C_p\Big(I+J\Big).
    \end{split}
    \label{eq:15}
    \end{equation}
We will see that
    \begin{equation}
    \label{eq: estimate I,J}
    I\leq C_1|t'-t|^{\beta_I p}, \qquad J\leq C_2|x'-x|^{\beta_J p},
    \end{equation}
    where $\beta_I,\beta_J>0$ are two positive constants.

    \medskip

    To start with, we have that
    \begin{equation*}
    \begin{split}
    I &  \leq C_p\Big( |I_0(t',x')-I_0(t,x')|^p \\
     & \quad + \mathrm{E}\Big[\Big| \int_{0}^{t'}\int_{\R} G_{t'-s}(x'-y) W^{H_n}(ds,dy)-\int_{0}^{t}\int_{\R} G_{t-s}(x'-y) W^{H_n}(ds,dy) \Big|^p\Big] \\
     & \quad + \mathrm{E}\Big[ \Big|  \int_{0}^{t'}\int_{\R} G_{t'-s}(x'-y) b(u^{H_n}(s,y)) dy\,ds
    - \int_{0}^{t}\int_{\R} G_{t-s}(x'-y) b(u^{H_n}(s,y)) dy\,ds\Big|^p \Big] \Big) \\
     & =: C_p( I_1+I_2+I_3).
    \end{split}
    \end{equation*}
    Regarding $I_1$, it is known from \cite{1/4<H<1/2}, Theorem 3.7,
    that, for a $\alpha$-H\"older continuous
    initial condition, it holds
    \begin{equation}
    I_1\leq C|t'-t|^{\frac{\alpha p}{2}}\leq C|t'-t|^{\frac{(\inf_ n H_n) p}{2}}.
    \label{eq:14}
    \end{equation}
Next, by step 1 in the proof of Theorem \ref{th: convergence linear
additive case}, we clearly obtain that
\begin{equation}
    I_2 \leq C |t'-t|^{\frac{H_n p}{2}} \leq  C |t'-t|^{\frac{(\inf_n
    H_n) p}{2}}.
    \label{eq:13}
\end{equation}
    It remains to estimate $I_3$.
    First, in the first summand of $I_3$ we perform the change of variables $s'=s-(t'-t)$, so that
    we obtain $I_3\leq C_p(I_{3,1}+I_{3,2})$, where
    $$I_{3,1}:=\E{\Big|\int_{-(t'-t)}^{0}\int_{\R} G_{t-s'}(x'-y) b(u^{H_n}(s'+(t'-t),y))ds'dy\Big|^p}$$
    and
    $$I_{3,2}:=\E{\Big|\int_{0}^{t}\int_{\R} G_{t-s}(x'-y) \Big(b(u^{H_n}(s+(t'-t),y))-b(u^{H_n}(s,y))\Big)dy\,ds\Big|^p}.$$
    Clearly, $I_{3,1}\leq C|t'-t|^{p}$ by  H\"older inequality,   Lemma \ref{lemma: uniform boundedness in L^p(Omega)} and the
    linear growth of $b$. For $I_{3,2}$, we have that
    \begin{equation*}
    \begin{split}
    I_{3,2} & = \E{\Big|\int_{0}^{t}\int_{\R} G_{t-s}(x'-y) \Big(b(u^{H_n}(s+(t'-t),y))-b(u^{H_n}(s,y))\Big)dy\,ds\Big|^p} \\
     & \leq C\,\E{\int_{0}^{t}\int_{\R} G_{t-s}(x'-y) \Big|u^{H_n}(s+(t'-t),y))-u^{H_n}(s,y)\Big|^p dy\,ds} \\
    & \leq  C\,\int_{0}^{t}\int_{\R} G_{t-s}(x'-y)  \Big(\sup_{n\geq 1} \ \sup_{y\in \R}\E{\Big|u^{H_n}(s+(t'-t),y))-u^{H_n}(s,y)\Big|^p}\Big)
    dy\,ds \\
     & =  C\,\int_{0}^{t}  \sup_{n\geq 1} \ \sup_{y\in
     \R}\E{\Big|u^{H_n}(s+(t'-t),y))-u^{H_n}(s,y)\Big|^p} ds.
    \end{split}
    \end{equation*}
    This latter estimate, together with \eqref{eq:14} and
    \eqref{eq:13} and the very definition of $I$, let us infer that
    \begin{equation*}
    \label{eq: I estimate before gronwall}
    \begin{split}
    & \sup_{n\geq 1} \ \sup_{x\in\R}
    \E{|u^{H_n}(t+(t'-t),x)-u^{H_n}(t,x)|^p} \\
    &  \qquad  \leq C_1 |t'-t|^{\beta_I p}
     + C_2 \int_{0}^{t}  \sup_{n\geq 1} \ \sup_{y\in \R}\E{\Big|u^{H_n}(s+(t'-t),y))-u^{H_n}(s,y)\Big|^p }
    ds,
    \end{split}
    \end{equation*}
    where the constants $C_1$ and $C_2$ do not depend on $H_n$ and
    $\beta_I=\frac 12 \inf_n H_n$. Hence, by Gr\"onwall lemma, we obtain the
    desired estimate for $I$ (see \eqref{eq: estimate I,J}).

\medskip

Let us now deal with the term $J$ in \eqref{eq:15}. Assume that
$x'=x+h$, for some $h>0$. We have
    \begin{equation}
    \begin{split}
    & \E{|u^{H_n}(t,x+h) - u^{H_n}(t,x)|^p} \leq C_p \Big( |I_0(t,x+h)-I_0(t,x)|^p \\
    & \qquad + \mathrm{E}\Big[\Big| \int_{0}^{t}\int_{\R} G_{t-s}(x+h-y) W^{H_n}(ds,dy)-\int_{0}^{t}\int_{\R} G_{t-s}(x-y) W^{H_n}(ds,dy) \Big|^p\Big] \\
    & \qquad + \mathrm{E}\Big[ \Big|  \int_{0}^{t}\int_{\R} G_{t-s}(x+h-y) b(u^{H_n}(s,y)) dy\,ds
     - \int_{0}^{t}\int_{\R} G_{t-s}(x-y)  b(u^{H_n}(s,y)) dy\,ds\Big|^p \Big] \Big) \\
     & \qquad =: J_1+J_2+J_3.
    \end{split}
    \end{equation}
    By \cite{1/4<H<1/2}, Theorem 3.7, and step 1 in the proof of Theorem \ref{th: convergence linear additive
    case}, we get, respectively,
    \begin{equation}
    J_1\leq  C \, h^{(\inf_n H_n)p} \qquad \text{and}\qquad J_2\leq  C \, h^{(\inf_n
    H_n)p}.
    \label{eq:17}
    \end{equation}
    In order to tackle the term $J_3$,
    we perform the change of variable $y'=y-h$ in its first summand,
    yielding
    \begin{equation*}
    \begin{split}
    J_3   = \mathrm{E}\Big[ \Big|  \int_{0}^{t}\int_{\R} G_{t-s}(x-y') b(u^{H_n}(s,y'+h)) dy'\,ds
     - \int_{0}^{t}\int_{\R} G_{t-s}(x-y) b(u^{H_n}(s,y)) dy\,ds\Big|^p
     \Big].
    \end{split}
    \end{equation*}
    Then, renaming the variable $y'$ as $y$, we have
    \begin{equation*}
    \begin{split}
    J_3   & = \mathrm{E}\Big[ \Big|  \int_{0}^{t}\int_{\R} \Big(b(u^{H_n}(s,y+h))- b(u^{H_n}(s,y))\Big)G_{t-s}(x-y) dy\,ds \Big|^p\Big] \\
     & \leq  C  \int_{0}^{t} \sup_{n\geq 1} \ \sup_{y\in \R} \E{\Big|u^{H_n}(s,y+h))- u^{H_n}(s,y))\Big|^p} ds. \\
    \end{split}
    \end{equation*}
    Putting together this bound and those of \eqref{eq:17}, we get
    \begin{equation*}
    \label{eq: J estimate before gronwall}
    \begin{split}
    & \sup_{n\geq 1}  \sup_{x\in\R}
    \E{|u^{H_n}(t,x+h)-u^{H_n}(t,x)|^p} \\
    & \qquad \qquad  \leq  C_1\, h^{\beta_J p }
    + C_2 \int_{0}^{t}  \sup_{n\geq 1} \ \sup_{y\in \R}\E{\Big|u^{H_n}(s,y+h))-u^{H_n}(s,y)\Big|^p }
    ds,
    \end{split}
    \end{equation*}
    where $\beta_J=\inf_n H_n$. By Gr\"onwall lemma,  we conclude
    that estimates \eqref{eq: estimate I,J} hold. Therefore, by Theorem \ref{th:
    centsov}, the family of laws of $\{u^{H_n}\}_{n\geq 1}$ is tight in
    $C([0,T]\times \R)$.

    \medskip

\textbf{Step $\mathbf{2}$:}
    This part of the proof is devoted to show the following uniform $L^2(\Omega)$-convergence:
        $$\sup_{H\in[a,b]}\sup_{(t,x)\in [0,T]\times\R}\E{|u_m^H(t,x)-u^H(t,x)|^2}\xrightarrow[m\to \infty]{}0.$$
    We remark that, indeed, the uniformity with respect to $(t,x)\in[0,T]\times \R$ will not
    be needed in step 3, but we obtain it for free thanks to our Gr\"onwall-type argument exhibited below.

    We argue as follows:
    \begin{equation}
    \label{eq: estimates}
    \begin{split}
    & \E{|u^H_m(t,x) -u^H(t,x)|^2}  \\
    & \quad \leq C \int_{0}^{t} \int_{\R} G_{t-s}(x-y) \E{|b_m(u^H_m(s,y))-b(u^H(s,y))|^2}dy\,ds  \\
    & \quad \leq  C \, \Big(\int_{0}^{t} \int_{\R} G_{t-s}(x-y) \E{|b_m(u^H_m(s,y))-b_m(u^H(s,y))|^2}dy\,ds  \\
    & \quad \quad  + \int_{0}^{t} \int_{\R} G_{t-s}(x-y) \E{|b_m(u^H(s,y))-b(u^H(s,y))|^2}dy\,ds \Big) \\
    & \quad \leq  C \, \Big( \int_{0}^{t} \int_{\R} G_{t-s}(x-y) \E{|u^H_m(s,y)-u^H(s,y)|^2}dy\,ds  \\
    & \quad \quad  + \int_{0}^{t} \int_{\R} G_{t-s}(x-y) \E{|b_m(u^H(s,y))-b(u^H(s,y))|^2
    1_{\{|u^H(s,y)|>m\}}}dy\,ds \Big). \\
    & \quad \leq  C \, \Big( \int_{0}^{t} \sup_{H\in[a,b]} \sup_{(s',y)\in [0,s]\times\R}\E{|u^H_m(s',y)-u^H(s',y)|^2} ds \\
    & \quad \quad  + \int_{0}^{t} \int_{\R} G_{t-s}(x-y)
    \E{|b_m(u^H(s,y))-b(u^H(s,y))|^4}^{\frac{1}{2}} \,
    \mathbb{P}(|u^H(s,y)>m|)^{\frac{1}{2}} \, dy\,ds \Big),
    \end{split}
    \end{equation}
    where in the progress we used the fact that $|b_m(u^H(s,y))-b(u^H(s,y))|=0$,
    whenever $|u^H(s,y)|\leq m$.

    A direct consequence of Lemma \ref{lemma: uniform boundedness in L^p(Omega)} is that $u^H$ is
    uniformly bounded in $L^p(\Omega)$, with respect to $H\in[a,b]$ and $(t,x)\in[0,T]\times \R$,
    for any $p\geq 2$, which means that  there exists a constant $M$ which depends only on $p$ and $T$
    such that
    \begin{equation}
    \label{eq: bound of the sup of L^p norms}
    \sup_{H\in[a,b]}\sup_{(t,x)\in [0,T]\times \R} \E{|u^H(t,x)|^p} \leq M.
    \end{equation}
    Hence, by Markov inequality,
    $$\mathbb{P}(|u^H(s,y)|>m)\leq \frac{\E{|u^H(s,y)|^2}}{m^2}\leq \frac{M}{m^2}.$$
    Note that the latter estimate is again uniform with respect to
    $H\in[a,b]$ and $(s,y)\in [0,T]\times \R$.
    Thus, going back to \eqref{eq: estimates} and using the linear growth of $b$ and
    \eqref{eq: bound of the sup of L^p norms}, we get
    \begin{equation}
    \begin{split}
    & \int_{0}^{t} \int_{\R} G_{t-s}(x-y) \E{|b_m(u^H(s,y))-b(u^H(s,y))|^4}^{\frac{1}{2}}
    \mathbb{P}(|u^H(s,y)>m|)^{\frac{1}{2}}dy\,ds \Big) \\
    & \quad \leq  \int_{0}^t \int_{\R} C \frac{M^{1/2}}{m} G_{t-s}(x-y)dy\,ds
    \leq  \int_{0}^t C \frac{M^{1/2}}{m} ds =:\frac{C}{m}.
    \end{split}
    \end{equation}
   We observe now that if on the left-hand side of \eqref{eq: estimates} we replace $t$ with any $t'\leq t$,
   the inequality would still hold exactly in the same way (indeed, the integrand on the right-hand
   side is positive, so it is increasing as a function of $t$). Therefore, we can infer that
    \begin{equation*}
    \begin{split}
    \sup_{H\in[a,b]}& \sup_{(t',x)\in [0,t]\times\R}    \E{|u^H_m(t',x) -u^H(t',x)|^2}  \\
     & \leq  \frac{C_1}{m}  + C_2 \int_{0}^{t} \sup_{H\in[a,b]}
     \sup_{(s',y)\in [0,s]\times\R}\E{|u^H_m(s',y)-u^H(s',y)|^2} ds.
    \end{split}
    \end{equation*}
    Then, Gr\"onwall lemma implies that
    $$\sup_{H\in[a,b]} \sup_{(t',x)\in [0,T]\times\R}     \E{|u^H_m(t',x) -u^H(t',x)|^2}
    \leq \frac{C}{m}\xrightarrow[m\to \infty]{} 0,$$
    which is what we wanted to show.

    \medskip

\textbf{Step $\mathbf{3}$:}
    We prove that the finite dimensional distributions of $u^{H_n}$ converge to those of $u^{H_0}$.
    Given a finite dimensional vector $\{(t_1,x_1),\dots,(t_k,x_k)\}$ and $f\in C_b(\R^k)$, we can write
    \begin{equation*}
    \begin{split}
    & \Big|\E{f\big(u^{H_n}(t_1,x_1), \dots,u^{H_n}(t_k,x_k)\big)-f\big(u^{H_0}(t_1,x_1),\dots,u^{H_0}(t_k,x_k)\big)}\Big| \\
     & \quad \leq  \Big| \E{f\big(u^{H_n}(t_1,x_1),\dots,u^{H_n}(t_k,x_k)\big)-
     f\big(u_m^{H_n}(t_1,x_1),\dots,u_m^{H_n}(t_k,x_k)\big)}  \Big| \\
     & \quad \quad +\Big| \E{f\big(u_m^{H_n}(t_1,x_1),\dots,u_m^{H_n}(t_k,x_k)\big)-
     f\big(u_m^{H_0}(t_1,x_1),\dots,u_m^{H_0}(t_k,x_k)\big)}  \Big| \\
    &  \quad \quad  +\Big| \E{f\big(u_m^{H_0}(t_1,x_1),\dots,u_m^{H_0}(t_k,x_k)\big)-
    f\big(u^{H_0}(t_1,x_1),\dots,u^{H_0}(t_k,x_k)\big)}  \Big| \\
     & \quad =: I_1(m,n)+I_2(m,n)+I_3(m).
    \end{split}
    \end{equation*}
    Assume that $f:\R^k\to\R$ is Lipschitz continuous with Lipschitz constant $L_f$
    (we can always restrict to the class of Lipschitz continuous functions to verify weak
    convergence). Then, for all $H\in [a,b]$,
    \begin{equation}
    \begin{split}
    & \sup_{H\in[a,b]}  \Big| \E{f(u^{H}(t_1,x_1),\dots,u^{H}(t_k,x_k))-f(u_m^{H}(t_1,x_1),\dots,u_m^{H}(t_k,x_k))}  \Big| \\
     & \quad \leq  \sup_{H\in[a,b]} \E{\Big|f(u^{H}(t_1,x_1),\dots,u^{H}(t_k,x_k))-f(u_m^{H}(t_1,x_1),\dots,u_m^{H}(t_k,x_k))\Big|}  \\
     & \quad \leq \sup_{H\in[a,b]}L_f \E{\Big(  \sum_{j=1}^{k} |u^H_m(t_j,x_j)-u^H(t_j,x_j)|^2  \Big)^{1/2}} \\
    & \quad \leq L_f \sup_{H\in[a,b]} \Big( \E{  \sum_{j=1}^{k} |u^H_m(t_j,x_j)-u^H(t_j,x_j)|^2 }\Big) ^{1/2} \\
    & \quad = L_f \sup_{H\in[a,b]} \Big( \sum_{j=1}^{k} \E{|u^H_m(t_j,x_j)-u^H(t_j,x_j)|^2 }\Big)^{1/2} \\
    & \quad \leq L_f k^{\frac12} \Big( \sup_{H\in[a,b]} \sup_{(t,x)\in [0,T]\times \R}
    \E{|u^H_m(t,x)-u^H(t,x)|^2 }\Big)^{1/2},
    \end{split}
    \end{equation}
    where the last term converges to 0 as $m\rightarrow \infty$ thanks to
    step 2, and taking into account that we are considering an arbitrary but fixed number of terms $k$.
    Hence, for any $\varepsilon>0$, there exists $m_0\geq 1$ such that, for all $m\geq m_0$, we have
    $$ \sup_{n\geq 1} \Big(I_1(m,n)+I_3(m)\Big) \leq \frac{\varepsilon}{2}.$$
    In particular, we have
    \begin{equation*}
    \begin{split}
    \Big|\E{f(u^{H_n}(t_1,x_1), \dots,u^{H_n}(t_k,x_k))-f(u^{H_0}(t_1,x_1),\dots,u^{H_0}(t_k,x_k))}\Big|
    \leq  I_2(m_0,n) +\frac{\varepsilon}{2}.
    \end{split}
    \end{equation*}
    Finally, it is sufficient to observe that the convergence
    \eqref{eq:20} implies the corresponding convergence of the finite dimensional distributions,
    and thus for some $n_0\geq 1$ we have that, for all $n\geq n_0$, it holds
    $I_2(m_0,n)<\frac{\varepsilon}{2}$.
    Therefore,
    $$\Big|\E{f(u^{H_n}(t_1,x_1), \dots,u^{H_n}(t_k,x_k))-f(u^{H_0}(t_1,x_1),\dots,u^{H_0}(t_k,x_k))}\Big|<\varepsilon,$$
    where $\varepsilon$ can be taken arbitrary small. This concludes the proof of Theorem \ref{thm:9}
    for the stochastic heat equation \eqref{eq: heat} in the case
    of a general Lipschitz continuous drift $b$.
    \qed


\section*{Acknowledgement}

Research supported by the grant MTM2015-67802P (Ministerio de Econom\'ia y Competitividad).


\end{document}